\documentclass[a4paper,10pt]{amsart}
\usepackage[utf8]{inputenc}
\usepackage{amssymb}
\usepackage{amsmath}
\usepackage{bbm}
\usepackage{mathrsfs}
\usepackage[all]{xy}
\usepackage{manfnt}
\usepackage{pigpen}
\usepackage{bm}
\usepackage{graphicx}

\usepackage{geometry}
\usepackage{tikz}
\title{On the functors associated with beaded open Jacobi diagrams}
\author{Christine Vespa}
\address{Universit\'e de Strasbourg, CNRS, IRMA UMR 7501, F-67000 Strasbourg, France.}
\email{vespa@math.unistra.fr}
\urladdr{http://irma.math.unistra.fr/~vespa/}

\thanks{This work was partially supported by the ANR Project {\em ChroK}, {\tt 
ANR-16-CE40-0003}, ANR Project AlMaRe {\tt ANR-19-CE40-0001-01} and ANR Project HighAGT {\tt ANR-20-CE40-0016}}
\newtheorem{THM}{Theorem}

\newtheorem{PROP}[THM]{Proposition}
\newtheorem{thm}{Theorem}[section]
\newtheorem{prop}[thm]{Proposition}
\newtheorem{cor}[thm]{Corollary}
\newtheorem{lem}[thm]{Lemma}

\theoremstyle{definition}

\theoremstyle{remark}
\newtheorem{rem}[thm]{Remark}

\renewcommand{\epsilon}{\varepsilon}
\renewcommand{\theta}{\vartheta}

\newcommand{\f}{\mathcal{F}}

\newcommand{\nat}{\mathbb{N}}

\newcommand{\gr}{\mathbf{gr}}

\newcommand{\A}{\mathfrak{a}}

\newcommand{\ev}{\mathbb{K}\text{-}\mathrm{Mod}}
\numberwithin{equation}{section}
\begin{document}
%\linenumbers
%\setpagewiselinenumbers

\begin{abstract}
Morphisms in the linear category $\mathbf{A}$ of Jacobi diagrams in handlebodies give rise to interesting contravariant functors on the category $\gr$ of finitely-generated free groups, encoding part of the composition structure of the category $\mathbf{A}$. These functors correspond, via an equivalence of categories given by Powell, to functors given by beaded open  Jacobi diagrams. We study the polynomiality  of these functors and whether they are outer functors. These results are inspired by and generalize previous results obtained by Katada.
\end{abstract}

\maketitle
\section{Introduction}
In \cite{HM}, Habiro and Massuyeau extend the Kontsevich integral to construct a functor from the category of bottom tangles in handlebodies to the  linear category $\mathbf{A}$ of Jacobi diagrams in handlebodies. This category has $\mathbb{N}$ as objects and for $n,m \in \mathbb{N}$, a generator of the vector space $\mathbf{A}(n,m)$ can be represented by a Jacobi diagram whose edges are oriented, such that each univalent vertex is embedded into the interior of the $1$-manifold $X_m$, consisting of $m$ arcs  and where we have beads coloured with elements of the free group of rank $n$ on edges. For example, for $F_3=\langle x_1, x_2, x_3 \rangle$,
the following is a non-zero element of $\mathbf{A}(3,2)$:
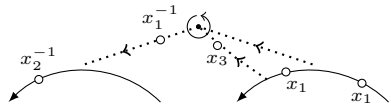
\[
	\vcenter{\hbox{\begin{tikzpicture}[baseline=1.8ex,scale=0.5]
	\draw[->,>=latex]  (4,0) to[bend right=45] (0,0);
		\draw[->,>=latex]  (10,0) to[bend right=45] (6,0);
			\draw[thick, dotted][-To,>=latex]   (5,2) to (3, 1.33);
				\draw[thick, dotted]  (3,1.33) to (2, 1);
				\draw[thick, dotted] [-To,>=latex]    (8,1) to (6.5, 1.5);
					\draw[thick, dotted] (6.5, 1.5) to (5, 2);
					\draw[thick, dotted]  (6,1.16) to (5, 2);
						\draw[thick, dotted][-To,>=latex]   (6.8,.6) to (6,1.16);
		\draw[fill=white] (.8,.6) circle (3pt);
		\draw (.8,.6) node[above] {$\scriptstyle{x_2^{-1}}$};
			\draw[fill=white] (4,1.65) circle (3pt);
		\draw (4,1.7) node[above] {$\scriptstyle{x_1^{-1}}$};
		\draw[fill=white] (5.5,1.5) circle (3pt);
		\draw (5.5, 1.5) node[below] {$\scriptstyle{x_3}$};
	\draw[fill=white] (7.3,0.8) circle (3pt);
		\draw (7.6, 0.8) node[below] {$\scriptstyle{x_1}$};
	\draw[fill=white] (9.3,0.5) circle (3pt);
		\draw (9.3, 0.5) node[below] {$\scriptstyle{x_1}$};
\draw[->] (4.9,2) +(80:.3) arc(-260:80:.3);
\draw[fill=black] (5,2) circle (1.5pt);
		\end{tikzpicture}}}
		\]

The definition of the composition in the linear category $\mathbf{A}$ is natural from the geometric point of view. However, it is quite complicated to understand it algebraically and this paper sheds some light on this. 

To study the composition in a category, we can fix an object $n$ in the category and look at the composition of the morphisms from $n$ with any morphism of the category.  In our setting, this corresponds to studying, for $n$ an object of $\mathbf{A}$, the linear functor $\mathbf{A}(n,-): \mathbf{A} \to \ev$ where $\mathbf{A}(-,-)$ denotes the $\mathbb{K}$-vector space of morphisms in $\mathbf{A}$, for $\mathbb{K}$ a field of characteristic zero. 
These functors being still too complicated to study, we restrict them to the subcategory $\mathbf{A}_0$ of $\mathbf{A}$ which is equivalent, by \cite[p. 630]{HM}, to the $\mathbb{K}$-linearization of the opposite of the category $\gr$ of finitely-generated free groups. 
This gives rise to functors 
$$\mathbf{A}(n,-): \gr^{op} \to \ev$$
encoding the composition of morphisms in $\mathbf{A}$ from $n$ with a morphism in the subcategory $\mathbf{A}_0$. The graduation by the degree $d$ of the Jacobi diagrams defines subfunctors of $\mathbf{A}(n,-)$:
$$\mathbf{A}_d(n,-): \gr^{op} \to \ev$$

For $n=0$, these functors have been studied by Katada in \cite{Katada, KatadaII}. Katada shows that  $\mathbf{A}_d(0,-)$ is a polynomial functor of degree $2d$ which is an outer functor: i.e. for all $m\in \nat$, the inner automorphisms act trivially on $\mathbf{A}_d(0,m)$. She also gives the complete structure of the functors $\mathbf{A}_d(0,-)$ for $d \in \{1, 2, 3\}$ and, for general $d$, a decomposition of the functor $\mathbf{A}_d(0,-)$ into indecomposables.

The aim of this paper is to study, more generally, the functors $\mathbf{A}_d(n,-): \gr^{op} \to \ev$ for $d, n \in \mathbb{N}$.

Our first result shows that the polynomiality of the functors $\mathbf{A}_d(0, -)$ is exceptional.
\begin{PROP} [Proposition \ref{poly-A_d(n)}]
For $d,n \in \mathbb{N}$, the functor $\mathbf{A}_d(n, -):\gr^{op} \to \ev$ is polynomial iff $n=0$. 
\end{PROP}
Using the $\gr^{op}$-graduation of $\mathbf{A}(n, m)$ introduced by Habiro and Massuyeau in \cite{HM}, we obtain, in Proposition \ref{sous-foncteur}, a  subfunctor $\mathbf{A}(n,-)_{\mathbf{0}}$ of $\mathbf{A}(n,-)$, satisfying $\mathbf{A}(0,-)_{\mathbf{0}}=\mathbf{A}(0,-)$. The generators of $\mathbf{A}(n,m)_{\mathbf{0}}$ are those of $\mathbf{A}(n,m)$ which can be represented by a Jacobi diagram on $X_m$ without beads on $X_m$ (but there may be beads on the Jacobi diagram). For example, for $F_3=\langle x_1, x_2, x_3 \rangle$,
the following represents a non-zero element of $\mathbf{A}(3,2)_{\mathbf{0}}$:

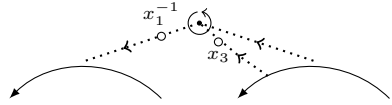
\[
	\vcenter{\hbox{\begin{tikzpicture}[baseline=1.8ex,scale=0.5]
	\draw[->,>=latex]  (4,0) to[bend right=45] (0,0);
		\draw[->,>=latex]  (10,0) to[bend right=45] (6,0);
			\draw[thick, dotted][-To,>=latex]   (5,2) to (3, 1.33);
				\draw[thick, dotted]  (3,1.33) to (2, 1);
				\draw[thick, dotted] [-To,>=latex]    (8,1) to (6.5, 1.5);
					\draw[thick, dotted] (6.5, 1.5) to (5, 2);
					\draw[thick, dotted]  (6,1.16) to (5, 2);
						\draw[thick, dotted][-To,>=latex]   (6.8,.6) to (6,1.16);
					\draw[fill=white] (4,1.65) circle (3pt);
		\draw (4,1.7) node[above] {$\scriptstyle{x_1^{-1}}$};
		\draw[fill=white] (5.5,1.5) circle (3pt);
		\draw (5.5, 1.5) node[below] {$\scriptstyle{x_3}$};
\draw[->] (4.9,2) +(80:.3) arc(-260:80:.3);
\draw[fill=black] (5,2) circle (1.5pt);
		\end{tikzpicture}}}
		\]

The graduation by the degree $d$ of the Jacobi diagrams defines a subfunctor $\mathbf{A}_d(n,-)_{\mathbf{0}}$ of $\mathbf{A}(n,-)_{\mathbf{0}}$. Considering the subspace of $\mathbf{A}_d(n,m)_{\mathbf{0}}$ generated by the Jacobi diagrams having at least $t$ trivalent vertices, we obtain subfunctors $\mathbf{A}_d^t(n,-)_{\mathbf{0}}$ of $\mathbf{A}_d(n,-)_{\mathbf{0}}$ defining a filtration:
\begin{equation} \label{filtration-intro}
0=\mathbf{A}^{2d}_{d}(n, -)_{\mathbf{0}}\subset \ldots \subset \mathbf{A}^{1}_d(n, -)_{\mathbf{0}}\subset \mathbf{A}^{0}_d(n, -)_{\mathbf{0}}= \mathbf{A}_d(n, -)_{\mathbf{0}}
\end{equation}
corresponding, for $n=0$, to the filtration considered by Katada in \cite{Katada}.

These functors satisfy the following:
\begin{THM} [Theorem \ref{poly_1}]
For $n \in \mathbb{N}$ and $d\geq 1$, the functor $\mathbf{A}_{d}(n,-)_{\mathbf{0} }: \gr^{op} \to \ev$ is polynomial of degree $2d$ and the filtration (\ref{filtration-intro}) corresponds to the polynomial filtration.
\end{THM}

However, contrary to the result of Katada, the functors $\mathbf{A}_d(n,-)_{\mathbf{0} }$  are rarely outer functors:
\begin{THM} [Theorem \ref{Thm-outre}]
For $d,n \in \nat$, the functor $\mathbf{A}_d(n,-)_{\mathbf{0} }$ is an outer functor iff $n=0$ or $d=0$.
\end{THM}

We obtain a description of the functor $\mathbf{A}_1(n,-)_{\mathbf{0} }$ generalizing that of the functor $\mathbf{A}_1(0,-)$ given in \cite[Section 4]{Katada}. Let $\A:\gr \to \ev$ be the abelianization functor, $\mathcal{P}_2:\gr \to \ev$ the second Passi functor  (see Section \ref{Action-S_2}) and $(-)^{\#}: \f(\gr; \mathbb{K})^{op} \to  \f(\gr^{op}; \mathbb{K})$ the duality functor, we have: 
\begin{PROP} (Proposition \ref{A11})
For $n \in \nat$, we have a natural equivalence
$$\mathbf{A}_1(n,-)_{\mathbf{0} }\simeq \mathcal{P}_2^{\#} \underset{\mathfrak{S}_2}{\otimes} \mathbb{K}[F_n],$$
where the action of $\mathfrak{S}_2$ on $\mathbb{K}[F_n]$ is given by taking the inverse in $F_n$: $v\mapsto v^{-1}$ and the action of $\mathfrak{S}_2$ on $\mathcal{P}_2^{\#}$ is given in Section \ref{Action-S_2}.
In particular, we have $\mathbf{A}_1(0,-) \simeq S^2 \circ \A^{\#}$.
\end{PROP}

In Section \ref{Ad0} we give another proof, based on \cite{PV}, of \cite[Theorem 10.1]{KatadaII}, giving a direct sum decomposition of the functor $\mathbf{A}_d(0,-)$ in the category of functors on $\gr^{op}$.

\vspace{.5cm}
One of the main ingredients of this paper is the use of the equivalence of categories given by Powell in \cite{ P-21-2}:
$$\alpha^{-1}: \f_\omega(\gr^{op};\mathbb{K})\xrightarrow{\simeq} \f_{\mathcal{L}ie}$$
where $\f_\omega(\gr^{op};\mathbb{K})$ is the category of analytic functors on $\gr^{op}$ and $\f_{\mathcal{L}ie}$ is the category of $\mathbb{K}$-linear functors from the linear \texttt{PROP} associated with the  operad $\mathcal{L}ie$ to $\ev$ (see Section \ref{CatLie} for further details). It turns out that the polynomial filtration of a functor is easier to understand in the category $\f_{\mathcal{L}ie}$ than in the category $\f_\omega(\gr^{op};\mathbb{K})$  (see  \cite{P-21-2} and Section \ref{CatLie}). It is also easier to show that the action of inner automorphisms is trivial in the category $\f_{\mathcal{L}ie}$ (see  \cite{P-21} and Section \ref{Rappels-G}). The proofs of the previous results are based on the computation of the functor $\alpha^{-1}(\mathbf{A}_d(n,-)_{\mathbf{0} })$. 
In Section \ref{JFm-section}, we introduce the $\mathbb{K}$-vector space  $J_d^{F_n}(m)$ which is the quotient by the AS and the IHX relations, of the $\mathbb{K}$-vector space generated by equivalence classes of open Jacobi diagrams $D$  whose edges are oriented and labelled by $F_n$ (represented by beads) and equipped with a bijection $\{\text{univalent vertices of } D\} \xrightarrow{\simeq } \{1, \ldots, m\}$. For example, the following is a non-zero element of  $J_2^{F_n}(3)$:

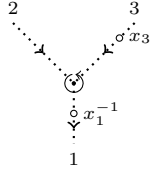
\[
	\vcenter{\hbox{\begin{tikzpicture}[baseline=1.8ex,scale=0.4]
				\draw[thick, dotted][-To,>=latex]   (2,2) to (2,0.5);
				\draw[thick, dotted]  (2,0.5) to (2,0);
				\draw[thick, dotted] [-To,>=latex]    (4,4) to (3,3);
					\draw[thick, dotted] (3,3) to (2, 2);
					\draw[thick, dotted] [-To,>=latex]    (0,4) to (1,3);
					\draw[thick, dotted] (1,3) to (2, 2);
	\draw[fill=white] (3.5,3.5) circle (3pt);
		\draw (3.5,3.5) node[right] {$\scriptstyle{x_3}$};
		\draw[fill=white] (2,1) circle (3pt);
		\draw (2,1) node[right] {$\scriptstyle{x_1^{-1}}$};
		\draw (2,0) node[below] {$\scriptstyle{1}$};
		\draw (0,4) node[above] {$\scriptstyle{2}$};
			\draw (4,4) node[above] {$\scriptstyle{3}$};
\draw[->] (1.9,2) +(80:.3) arc(-260:80:.3);
\draw[fill=black] (2,2) circle (1.5pt);
		\end{tikzpicture}}}
		\]

The generators of $J_d^{F_n}(m)$ are called  \textit{$F_n$-beaded open Jacobi diagrams}. The correspondance between the AS relation and the antisymmetry relation for Lie algebras and the IHX relation and the Jacobi relation for Lie algebras implies that this defines a functor $J_d^{F_n}$ in $\f_{\mathcal{L}ie}$
(see Proposition \ref{Foncteurs-J_d}). We have the following:

\begin{THM} [Theorem \ref{thm-J_d}]
For $n,d \in \mathbb{N}$, we have an equivalence of functors in $\f_{\mathcal{L}ie}$:
$$\alpha^{-1}(\mathbf{A}_d(n,-)_{\mathbf{0} })\simeq J_d^{F_n}.$$
\end{THM}
The title of this paper reflects the fact that the functors $J_d^{F_n}$ are much easier to study than the functors $\mathbf{A}_d(n,-)_{\mathbf{0} }$. A more in-depth study of the functors $J_d^{F_n}$ will be given in another paper.

\vspace{.5cm}
\textbf{Acknowledgement:} This work was inspired by Mai Katada's papers \cite{Katada, KatadaII}, where the natural appearance of outer polynomial functors intrigued the author of this paper and motivated her to take an interest in the category of Jacobi diagrams in handlebodies. The author is grateful to Mai Katada for communicating her papers. The author is also grateful to Geoffrey Powell for discussions on his papers \cite{P-21, P-21-2} and helpful comments on the previous versions of the manuscript.
\tableofcontents

\noindent
\textbf{Notation.} Denote by:\\
$\mathbb{K}$ a field of characteristic $0$.\\
$\ev$ the category of $\mathbb{K}$-vector spaces.\\
For $n\geq 0$:
\begin{itemize}
\item $\mathbf{n}=\{1, \ldots, n\}$;
\item $\mathfrak{S}_n$ is the symmetric group on $n$ letters;
\item $X_n$ is the oriented $1$-manifold consisting of $n$ arc components;
\item $F_n=\langle x_1, \ldots, x_n \rangle$ is the free group of rank $n$. The trivial group is denoted by $\{1\}$.
\end{itemize}

For generalities on Jacobi diagrams we refer the reader to \cite[Chapter 5]{CDM}.

%%%%%%%%%%%%%%%%
\section{Functors on $\gr^{op} $}

\subsection{Generalities on $\gr $} \label{gr-op}

Let $\gr $ be the category of finitely-generated free groups. This category is essentially small, with skeleton given by $\nat$, where $n\in \nat$ corresponds to the free group $F_n$ of rank $n$. For clarity, we will sometimes denote the object $n$ by $F_n$. The object $0=F_0=\{1\}$ is a null-object in $\gr $. For $n,m$ objects of $\gr $, we denote by $\mathbf{0}: n \to m$ the composition $n \to 0 \to m$. Explicitly, $\mathbf{0}$ is the homomorphism $F_m \to F_n$ sending each generator to $1$. 

The category $\gr $ is a PROP for the symmetric strict monoidal structure given by the free product.

By Pirashvili's result \cite{P}, the PROP $\gr $ is isomorphic to the free symmetric monoidal category generated by a commutative Hopf monoid. In other words, the morphisms of $\gr $ are generated by the permutations groups in $\gr(n,n)$ for $n \in \nat$ and the following homomorphisms:
\begin{enumerate}
\item $m_1: 1 \to 0$ corresponding to $F_1 \to \{1\}$;
\item $m_2: 1 \to 2$ corresponding to $F_1 \to F_2$ sending the generator $x$ of $F_1$ to $x_1x_2$;
\item $m_3: 0 \to 1$ corresponding to $\{1\} \to F_1$;
\item $m_4: 1 \to 1$ corresponding to $F_1 \to F_1$ sending the generator $x$ of $F_1$ to $x^{-1}$;
\item $m_5: 2 \to 1$ corresponding to $F_2 \to F_1$ sending $x_1$ to $x$ and $x_2$ to $x$ (this is the folding map).
\end{enumerate}

\subsection{Generalities on functors on $\gr$ and $\gr^{op} $}
We denote by $\f(\gr; \mathbb{K})$ (resp. $\f(\gr^{op}; \mathbb{K})$) the category of functors from $\gr$ (resp. $\gr^{op}$) to $\ev$. These categories are abelian.

A functor $M: \gr \to \ev$ (resp. $N: \gr^{op} \to \ev$) is said to be \textit{reduced} if $M(0)=0$ (resp. $N(0)=0$).

Let $P_n: \gr \to \ev$ be the functor $\mathbb{K}[\gr(n,-)]$; $\{P_n, n \in \mathbb{N}\}$ is a set of projective generators of the category $\f(\gr; \mathbb{K})$. By the Yoneda lemma, for $F: \gr \to \ev$, $\text{Hom}_{\f(\gr; \mathbb{K})}(P_n,F)\simeq F(n)$.

We denote by $\bar{P}_1$ the reduced part of $P_1$ i.e. $P_1\simeq \mathbb{K} \oplus \bar{P}_1$. For $G$ a free group $\bar{P}_1(G)$ is the $\mathbb{K}$-vector space underlying the augmentation ideal $IG$ of the $\mathbb{K}$-algebra $\mathbb{K}[G]$. Since $P_0=\mathbb{K}$, we have $\text{Hom}_{\f(\gr; \mathbb{K})}(P_1,F)\simeq F(0) \oplus \text{Hom}_{\f(\gr; \mathbb{K})}(\bar{P}_1,F)$, so 
\begin{equation}\label{End(P_1)}
\text{Hom}_{\f(\gr; \mathbb{K})}(\bar{P}_1,\bar{P}_1)\simeq \bar{P}_1(1)\simeq IF_1.
\end{equation}

Composition with vector space duality functor $(-)^{\#}:\ev \to \ev^{op}$ gives rise to a pair of adjoint functors, named the \textit{duality functors}
$$(-)^{\#}: \f(\gr; \mathbb{K})^{op} \to Func(\gr; \ev^{op})^{op} \xrightarrow{\simeq} \f(\gr^{op}; \mathbb{K})$$
$$(-)^{\#}:  \f(\gr^{op}; \mathbb{K}) \to Func(\gr^{op}; \ev^{op}) \xrightarrow{\simeq} \f(\gr; \mathbb{K})^{op}$$
where the last equivalences are given by the usual equivalence of categories $\f(\gr^{op}; \mathbb{K})^{op}\simeq Func(\gr; \ev^{op})$, where $Func(\gr; \ev^{op})$ is the category of functors from $\gr$ to the opposite of $\ev$.
The duality functors restrict to an equivalence of categories for functors taking finite dimensional values.

Let $\A: \gr \to \ev$ be the abelianization functor that sends a free group $G$ to $(G/[G,G])\underset{\mathbb{Z}}{\otimes} \mathbb{K}$.

The category of outer functors $ \f^{Out}(\gr^{op}; \mathbb{K})$ is the full subcategory of $\f(\gr^{op}; \mathbb{K})$ of functors $F$ such that, for each $n \in \nat$, inner automorphisms act trivially on $F(n)$. Outer functors were introduced in \cite[Section 10]{PV}.
Let $\Omega: \f(\gr^{op}; \mathbb{K}) \to \f^{Out}(\gr^{op}; \mathbb{K})$ be the left adjoint to the inclusion functor $ \f^{Out}(\gr^{op}; \mathbb{K}) \hookrightarrow  \f(\gr^{op}; \mathbb{K})$. The functor $\Omega$ is described explicitly in  \cite[Definition 11.5]{PV}.

\subsection{Polynomial and analytic functors on $\gr^{op} $} \label{rappels-poly}
Polynomial contravariant functors have been considered in  a general setting in \cite[Section 3.1]{HPV}. Here we recall the definitions for contravariant functors on $\gr$.

For $k \in \{1, \ldots, n\}$, let $i^{n}_{\hat{k}}: F_{n-1} \to F_{n}$
be the homomorphism given by
$$i^{n}_{\hat{k}}(x_i)=\left\{
    \begin{array}{ll}
       x_i & \mbox{if } i<k \\
       x_{i+1} & \mbox{if } i\geq k
    \end{array}
    \right.
    $$

The $n$-th cross-effect of a functor $N: \gr^{op} \to \ev$ is a functor  $\widetilde{cr}_{n}(N):  (\gr^{op})^{\times n} \to \ev$. Its evaluation on $F_1$ in each variable $\widetilde{cr}_{n}(N)(1,\ldots, 1)$ is equal to the kernel of the natural homomorphism 
$$N(F_{n}) \xrightarrow{(N(i^{n}_{\hat{1}}), \ldots, N(i^{n}_{\hat{n}}))} \overset{n}{\underset{k=1}{\bigoplus}}N(F_{n-1}).$$

In the examples, it is easier to compute cross-effects using the following equivalent description using a cokernel instead of a kernel.
For $k \in \{1, \ldots, n\}$, let $r^{n}_{\hat{k}}: F_{n} \to F_{n-1}$
be the homomorphism given by
$$r^{n}_{\hat{k}}(x_i)=\left\{
    \begin{array}{ll}
       x_i & \mbox{if } i<k \\
        1 & \mbox{if } i=k \\
       x_{i-1} & \mbox{if } i>k
    \end{array}
    \right.
    $$
$\widetilde{cr}_{n}(N)(1,\ldots, 1)$ is isomorphic to the cokernel of the natural homomorphism 
$$ \overset{n}{\underset{k=1}{\bigoplus}}N(F_{n-1})\xrightarrow{(N(r^{n}_{\hat{1}}), \ldots, N(r^{n}_{\hat{n}}))}N(F_{n}).$$

For $d \in \nat$, a  functor $N: \gr^{op} \to \ev$ is polynomial of degree at most $d$ if  $\widetilde{cr}_{d+1}(N)(1,\ldots, 1)=0$. Let $\f_d(\gr^{op}; \mathbb{K})$ be the full subcategory of polynomial functors of degree at most $d$. The forgetful functor $\f_d(\gr^{op}; \mathbb{K}) \to \f(\gr^{op}; \mathbb{K})$ has a right adjoint denoted by $\mathbf{p}_d$. For $N: \gr^{op} \to \ev$, the functor $\mathbf{p}_d(N)$ is the largest subfunctor of $N$ polynomial of degree $d$. Hence, a functor $N: \gr^{op} \to \ev$ admits a natural filtration, called the \textit{polynomial filtration} of $N$:
$$\mathbf{p}_0(N) \subset \mathbf{p}_1(N) \subset \ldots \subset \mathbf{p}_d(N) \subset \mathbf{p}_{d+1}(N) \subset \ldots \subset N.$$

A functor $N: \gr^{op} \to \ev$ is \textit{analytic} if it is the colimit of its subfunctors $\mathbf{p}_d(N)$. Let  $\f_\omega(\gr^{op};\mathbb{K})$ denote the full subcategory of $ \f(\gr^{op}; \mathbb{K}) $ of analytic functors.

For $d \in \nat$, we have a functor:
$\overline{cr}_{d}: \f_d(\gr^{op}; \mathbb{K})  \to \mathbb{K}[\mathfrak{S}_d]\text{-Mod}$, given on $N: \gr^{op} \to \ev$ by $\overline{cr}_{d}(N)=\widetilde{cr}_{d}(N)(1,\ldots, 1)$ where the action of $\mathfrak{S}_d$ is given by permuting the factors.

The functor $(\A^{\#})^{\otimes d}: \gr^{op} \to \ev$ is polynomial of degree $d$. The following Proposition is the analogue, for contravariant functors, of \cite[Proposition 6.9]{PV}.
\begin{prop}\cite[Proposition 7.20]{P-21-2} \label{filtration-poly}
For $d \in \nat$, the functor $\overline{cr}_{d}: \f_d(\gr^{op}; \mathbb{K})  \to \mathbb{K}[\mathfrak{S}_d]\text{-}\mathrm{Mod}$ has right adjoint given by 
$$M \mapsto ((\A^{\#})^{\otimes d} \otimes M)^{\mathfrak{S}_d}$$
where $\mathfrak{S}_d$ acts diagonally. This functor is exact and $(({\A^{\#}})^{\otimes d} \otimes M)^{\mathfrak{S}_d}$ is semi-simple of polynomial degree $d$. 

For  $N: \gr^{op} \to \ev$ there is a natural short exact sequence
$$0 \to \mathbf{p}_{d-1}(N) \to \mathbf{p}_d(N) \to ((\A^{\#})^{\otimes d} \otimes \overline{cr}_{d}( \mathbf{p}_dN))^{\mathfrak{S}_d}\to 0.$$

\end{prop}
%%%%%%%%%%%%%%%%%%%%%%
\subsection{Analytic functors on $\gr^{op}$ and $Cat\mathcal{L}ie$-modules} \label{CatLie}
Let $Cat\mathcal{L}ie$ be the  linear \texttt{PROP} associated with the  operad $\mathcal{L}ie$ \cite[Section 5.4.1]{LV}. Explicitly, $Cat\mathcal{L}ie$ is the $\mathbb{K}$-linear category such that  $\text{Ob}(Cat\mathcal{L}ie)=\mathbb{N}$ and 
$$Cat\mathcal{L}ie(m,n)=\underset{f\in \mathbf{Fin}(m,n)}{\bigoplus}\overset{n}{\underset{i=1}{\bigotimes}}\mathcal{L}ie(|f^{-1}(i)|)$$
where $\mathbf{Fin}$ is the category of finite sets. Since $\mathcal{L}ie$ is reduced (i.e. $\mathcal{L}ie(0)=0$) the sum can be taken over the surjections $\mathbf{m} \twoheadrightarrow \mathbf{n}$. For $m \in \nat$, $Cat\mathcal{L}ie(m,1)=\mathcal{L}ie(m)$ and for $m<n$, $Cat\mathcal{L}ie(m,n)=0$. Since $\mathcal{L}ie(1)=\mathbb{K}$, $Cat\mathcal{L}ie(m,m) \simeq \mathbb{K}[\mathfrak{S}_m]$.

Fix a generator $\mu \in \mathcal{L}ie(2)$. For $n\in \mathbb{N}$ and $i\in \{1, \ldots, n\}$, let $\mu_i^{n+1} \in Cat\mathcal{L}ie(n+1,n)$ be the morphism given by the set map $s_i^{n+1}: \mathbf{n+1} \to \mathbf{n}$ defined by $s_i^{n+1}(j)=j$ for $j<n+1$ and $s_i^{n+1}(n+1)=i$ and taking $1 \in \mathcal{L}ie(1)$ for the fibres of cardinal one and $\mu \in \mathcal{L}ie(2)$ for the fiber of cardinal $2$. The $\mathbb{K}$-linear category $Cat\mathcal{L}ie$ is generated (via linear combination  and composition) by the morphisms  $\mu_i^{n+1} \in Cat\mathcal{L}ie(n+1,n)$ and $Cat\mathcal{L}ie(n,n)\simeq \mathbb{K}(\mathfrak{S}_n)$ for $n \in \mathbb{N}$.

Note that a pointed version of $Cat\mathcal{L}ie$ with a shuffle condition on fibers intervenes  in \cite{HV}.

Let $\f_{\mathcal{L}ie}$ be the category of $\mathbb{K}$-linear functors from $Cat\mathcal{L}ie$ to $\ev$.
For $n \in \mathbb{N}$, $Cat\mathcal{L}ie(n,-): Cat\mathcal{L}ie \to \ev$ is a linear functor. By the enriched Yoneda lemma, for $F:Cat\mathcal{L}ie\to \ev$ a $\mathbb{K}$-linear functor, we have an isomorphism:
\begin{equation*}
\text{Hom}_{\f_{\mathcal{L}ie}}(Cat\mathcal{L}ie(n,-), F) \simeq F(n).
\end{equation*}
We deduce that the functors $Cat\mathcal{L}ie(n,-)$, for $n \geq 0$, are projective generators of $\f_{\mathcal{L}ie}$.

For $F \in \f_{\mathcal{L}ie}$ and $i \in \nat$, since $Cat\mathcal{L}ie(m,n)=0$   for $m<n$, $F$ admits a subfunctor $F_{\leq i}$ given by truncation, i.e. 
$$F_{\leq i}(n)=\left\lbrace\begin{array}{ll}
 F(n) & \text{if } n\leq i\\
 0 &  \text{if } n>i
 \end{array}
 \right.
$$
It follows that a functor $F \in \f_{\mathcal{L}ie}$ admits a natural filtration:
$$F_{\leq 0} \subset F_{\leq 1} \subset \ldots \subset F_{\leq d}  \subset F_{\leq d+1} \subset \ldots \subset F.$$

In \cite{P-21-2}, Powell gives an equivalence of categories between $\f_{\mathcal{L}ie}$ and $\f_\omega(\gr^{op};\mathbb{K})$. In particular, Powell constructs explicit exact functors:
$$\alpha: \f_{\mathcal{L}ie} \to  \f_\omega(\gr^{op};\mathbb{K})$$
$$\alpha^{-1}: \f_\omega(\gr^{op};\mathbb{K}) \to \f_{\mathcal{L}ie}$$
giving this equivalence.
By \cite[Corollary 8.9]{P-21-2}, for $F \in  \f_\omega(\gr^{op};\mathbb{K})$ we have an isomorphism:
\begin{equation}\label{alpha}
\alpha^{-1}(F)(d) \simeq \widetilde{cr_d} (\mathbf{p}_d F)(1, \ldots, 1).
\end{equation}

The category $Cat\mathcal{L}ie$ is easier to understand than the category $\gr^{op}$. For example, we have $Cat\mathcal{L}ie(i,j)=0$ for $i<j$.
It follows that it is easier to work with $Cat\mathcal{L}ie$-modules than with functors on $\gr^{op}$. In particular, the polynomiality of functors in $\f_\omega(\gr^{op};\mathbb{K})$ has an easy interpretation in $Cat\mathcal{L}ie$-modules: a functor $F  \in \f_\omega(\gr^{op};\mathbb{K})$ is polynomial of degree equal to $d$ iff $\alpha(F)(d)\not=0$ and $\alpha(F)(k)=0$ for $k>d$. Via the equivalence of categories, the polynomial filtration of a functor in $\f_\omega(\gr^{op};\mathbb{K})$ corresponds to the filtration given by the truncations of $Cat\mathcal{L}ie$-modules. More precisely, for $N: \gr^{op} \to \ev$  and $i\in \nat$ we have:
\begin{equation}\label{filtration-general}
\alpha^{-1}(\mathbf{p}_i(N))=(\alpha^{-1} N)_{\leq i}.
\end{equation}

To prove Theorem \ref{thm-J_d}, we will need the following explicit description of the functor $\alpha$ given in \cite[Theorem 9.17]{P-21-2}.
Let $Cat\mathcal{A}ss^u$ be the linear PROP associated with unital associative algebras; $Cat\mathcal{A}ss^u$ is the $\mathbb{K}$-linear category such that  $\text{Ob}(Cat\mathcal{A}ss^u)=\mathbb{N}$ and 
$$Cat\mathcal{A}ss^u(m,n)=\underset{f\in \mathbf{Fin}(m,n)}{\bigoplus}\overset{n}{\underset{i=1}{\bigotimes}}\mathcal{A}ss^u(|f^{-1}(i)|)$$
where $\mathbf{Fin}$ is the category of finite sets. More explicitly, a generator of $Cat\mathcal{A}ss^u(m,n)$ is represented by a set map $f\in \mathbf{Fin}(m,n)$ and an order of the elements of each fiber of $f$. We denote by $(s_i^{n+1}, i<n+1)$ (resp. $(s_i^{n+1}, n+1<i)$) the morphism in  $Cat\mathcal{A}ss(n+1,n)$  given by the set map $s_i^{n+1}: \mathbf{n+1} \to \mathbf{n}$ and the order $i<n+1$ (resp. $n+1<i$) on the fiber of cardinal $2$.  The morphism of operads $\mathcal{L}ie \to \mathcal{A}ss^u$ induces a functor $Cat\mathcal{L}ie \to Cat\mathcal{A}ss^u$ sending the morphism $\mu_i^{n+1} \in Cat\mathcal{L}ie(n+1,n)$ to $(s_i^{n+1}, i<n+1)-(s_i^{n+1}, n+1<i) \in Cat\mathcal{A}ss(n+1,n)$

By   \cite[Proposition 9.13]{P-21-2} the sets of morphisms in $Cat\mathcal{A}ss^u$  define a functor:
$$Cat\mathcal{A}ss^u: (Cat\mathcal{L}ie)^{op} \otimes  \mathbb{K}[\gr^{op}]  \to \ev.$$
In  \cite[Lemma A.2]{P-21-2}, the functor $Cat\mathcal{A}ss^u(i,-): \gr^{op}  \to \ev$, for $i$ an object of $(Cat\mathcal{L}ie)^{op}$, is described explicitly on the generators of $\gr$ recalled in Section  \ref{gr-op}.
By \cite[Theorem  9.17]{P-21-2}, $\alpha=Cat\mathcal{A}ss^u \underset{Cat\mathcal{L}ie}{\otimes }-$.  Let $\Sigma$ be the category of finite sets and bijections and $\f(\Sigma;\mathbb{K})$  the category of functors from $\Sigma$ to $\ev$. The obvious functor $\Sigma \to \gr^{op}$ induces a functor $ \mathcal{F}(\mathbf{gr}^{op}; \mathbb{K}) \to \f(\Sigma;\mathbb{K})$.  By \cite[Remark 9.18]{P-21-2},  for an object $F$ of $\mathcal{F}_{Lie}$ and $d$ an object of $\gr^{op}$, the functor in  $\f(\Sigma;\mathbb{K})$ associated  with $Cat\mathcal{A}ss^u \underset{Cat\mathcal{L}ie}{\otimes }F$ is given explicitly by:
\begin{equation}\label{alpha2}
\alpha(F)(d)\simeq \underset{i \in \mathbb{N}}{\bigoplus} \mathbb{K}\mathbf{Fin}(i,d) \underset{\mathfrak{S}_i}{\otimes} F(i).
\end{equation}

A $\mathbb{K}[\mathfrak{S}_d]$-module $M$ defines an  object of $\mathcal{F}_{Lie}$ which is $0$ for $n\neq d$ and $M$ on $d$. Such $Cat\mathcal{L}ie$-module will be said \textit{atomic} and will be denoted by $M[d]$.

For a $\mathbb{K}[\mathfrak{S}_d]$-module $M$, by \cite[Example 9.20]{P-21-2}, we have: 
\begin{equation}\label{atomic}
\alpha(M[d])= (\A^{\#})^{\otimes d} \underset{\mathfrak{S}_d}{\otimes} M.
\end{equation}
 This is in the image of the faithful embedding $\mathcal{F}_\omega(\mathbf{ab}^{op}; \mathbb{K}) \hookrightarrow \mathcal{F}_\omega(\mathbf{gr}^{op}; \mathbb{K})$, where $\mathbf{ab}$ is the category of finitely-generated free abelian groups. The category of analytic functors $\mathcal{F}_\omega(\mathbf{ab}^{op}; \mathbb{K})$ is semi-simple. More precisely we have an equivalence of categories:
 $$\mathcal{F}_\omega(\mathbf{ab}^{op}; \mathbb{K}) \simeq \f(\Sigma;\mathbb{K}).$$
The functor $\alpha^{-1}$ extends this equivalence of categories in the sense that we have a commutative diagram:
$$\xymatrix{
 \mathcal{F}_{Lie} \ar[r]^-{\alpha}_-{\simeq} & \mathcal{F}_\omega(\mathbf{gr}^{op}; \mathbb{K})\\
\f(\Sigma;\mathbb{K}) \ar[r]_-{\simeq} \ar@{^{(}->}[u]& \mathcal{F}_\omega(\mathbf{ab}^{op}; \mathbb{K}).  \ar@{^{(}->}[u]
 }$$
 
 \begin{rem}
 For $N: \gr^{op} \to \ev$, by Proposition \ref{filtration-poly} we can consider the graded functor $gr(N)$ associated with the filtered functor $N$, obtained by considering the polynomial filtration. We have $gr(N) =\underset{d \in \nat}{\bigoplus} ((\A^{\#})^{\otimes d} \otimes \overline{cr}_{d}( \mathbf{p}_dN))^{\mathfrak{S}_d}$ and $\alpha(gr(N))$ is the direct sum of atomic functors associated with the $\mathbb{K}[\mathfrak{S}_d]$-module $\overline{cr}_{d}( \mathbf{p}_dN)$ whereas $\alpha(N)$ \textit{is not}, in general, the direct sum of atomic functors. This illustrates the fact that, considering the graded associated with a functor, we lose much of the structure. 
 \end{rem}

\subsection{Outer $Cat\mathcal{L}ie$-modules} \label{Rappels-G}
 In \cite{P-21}, Powell gives a characterization of $Cat\mathcal{L}ie$-modules corresponding to outer functors via the equivalence of categories given in the previous section. These $Cat\mathcal{L}ie$-modules will be called \textit{outer $Cat\mathcal{L}ie$-modules}. We briefly recall Powell's result. 
 
 Let $\tau: \mathcal{F}_{Lie} \to \mathcal{F}_{Lie}$ be the shifting functor given by precomposition with $-+1: Cat\mathcal{L}ie \to Cat\mathcal{L}ie$.
 Let $\mu: \tau \to Id$ be the natural transformation defined as follows:
 for $F \in  \mathcal{F}_{Lie}$,   $\mu_F: \tau F \to F$ is given by the natural morphisms $(\mu_F)_n: \tau F(n)=F(n+1) \to F(n)$ induced by $\overset{n}{\underset{i=1}{\sum}}\mu_i^{n+1} \in Cat\mathcal{L}ie(n+1,n)$. Let $\mathcal{F}_{Lie}^{\mu}$ be the full subcategory of $\mathcal{F}_{Lie}$ of functors such that $\mu_F=0$. By \cite[Theorem 4.16]{P-21}, under the equivalence of categories $ \f_\omega(\gr^{op};\mathbb{K}) \simeq  \f_{\mathcal{L}ie}$, the full subcategory  $ \f^{Out}_\omega(\gr^{op};\mathbb{K})$ of  $ \f_\omega(\gr^{op};\mathbb{K}) $ is equivalent to $\mathcal{F}_{Lie}^{\mu}$.
 
 Let $(-)^\mu: \mathcal{F}_{Lie} \to \mathcal{F}_{Lie}^{\mu}$ be the functor given by $F^\mu:=coker(\mu_F)$. By \cite[Proposition 2.17]{P-21}, $(-)^\mu$ is the left adjoint to the inclusion $\mathcal{F}_{Lie}^{\mu} \hookrightarrow \mathcal{F}_{Lie}$ and so corresponds to the functor $\Omega: \f(\gr^{op}; \mathbb{K}) \to \f^{Out}(\gr^{op}; \mathbb{K})$ via the equivalence of categories $ \f_\omega(\gr^{op};\mathbb{K}) \simeq  \f_{\mathcal{L}ie}$.
 %%%%%%%%%%%%%%%%%%%%%%%%%%
\section{On the second Passi functor $\mathcal{P}_2$} \label{Action-S_2}

The contents of this Section will be used in the proof of Proposition \ref{A11}. As the results of this Section are of independent interest, we chose to dedicate a separate Section to them. The reader can skip this section on first reading.

Let $\mathcal{P}_2: \gr \to \ev$ be the functor defined by: $\mathcal{P}_2(F_n)={IF_n}/{(IF_n)^3}.$
The functor $\mathcal{P}_2$ is called the second Passi functor (see \cite{HPV, V_ext, PV}). It is the largest quotient of $\overline{P}_1$ that is polynomial of degree $2$.

The group $\mathfrak{S}_2$ acts on $\mathcal{P}_2$ by the following way: by (\ref{End(P_1)}), the element $[x_1^{-1}]-[1]$ of $IF_1$ corresponds to a natural transformation $\sigma$ in $\mathrm{End}_{\f(\gr; \mathbb{K})}(\overline{P}_1)$.
For $G \in \gr$, $\sigma_G: IG \to IG$ is given explicitly by: $\sigma_G([g]-[1])=[g^{-1}]-[1]$.  Since $\sigma^2=1$, $\sigma$ defines an action of $\mathfrak{S}_2$ on $\overline{P}_1$.  By composition with $ \overline{P}_1  \twoheadrightarrow \mathcal{P}_2$ we obtain 
$$\sigma \in \mathrm{Hom}_{\f(\gr)}( \overline{P}_1,\mathcal{P}_2)\simeq \mathrm{Hom}_{\f_2(\gr)}(\mathcal{P}_2,\mathcal{P}_2 )$$
where the last isomorphism is given by adjunction, so that $\mathfrak{S}_2$ acts on $\mathcal{P}_2$.

In the following lemma we give an explicit description of this action of $\mathfrak{S}_2$ on $\mathcal{P}_2$. 
Recall that  $\A(G) \simeq IG/(IG)^2$, for $G$ an object of $\gr$.

\begin{lem}\label{action-Passi}
The natural transformation $\sigma: \mathcal{P}_2\to  \mathcal{P}_2 $ restricts to a natural transformation
 $\sigma_{|\A^{\otimes 2} }: \A^{\otimes 2} \to  \A^{\otimes 2}$ given by the place permutation. The induced natural transformation
$\overline{\sigma}: \A \to  \A$ is given by $\overline{\sigma}(x)=-x$ for $G$ an object of $\gr$ and $x$ an element in  $\A(G)$.
\end{lem}

\begin{proof}
By \cite{V_ext, DPV}, $\mathcal{P}_2$ is a generator of $\mathrm{Ext}^1_{\f(\gr)}(\A, \A^{\otimes 2})$. The non-split short exact sequence 
\begin{equation}\label{ses-Passi}
\xymatrix{
 0 \ar[r] & \A^{\otimes 2} \ar[r]^-i& \mathcal{P}_2 \ar[r]^-p & \A \ar[r] & 0
  }
  \end{equation}
  gives rise, for $G$ an object of $\gr$, to an exact sequence
   $$\xymatrix{
 0 \ar[r] & IG/(IG)^2 \otimes IG/(IG)^2 \ar[r]^-i & IG/(IG)^3 \ar[r]^-p & IG/(IG)^2  \ar[r] & 0
  }$$
 For $x,y \in G$ we have
  $$\sigma_G\circ i \left( ([x]-[1]+(IG)^2)\otimes ([y]-[1]+(IG)^2) \right)=\sigma_G \left( ([x]-[1]).([y]-[1])+(IG)^3\right)$$
  $$=\sigma_G\left( ([xy]-[1])-([x]-[1])-([y]-[1])+(IG)^3\right)$$
  $$= ([y^{-1}x^{-1}]-[1])-([x^{-1}]-[1])-([y^{-1}]-[1])+(IG)^3$$
  $$=([y^{-1}]-[1]).([x^{-1}]-[1])+(IG)^3$$
  $$=i\left(([y^{-1}]-[1]+(IG)^2)\otimes ([x^{-1}]-[1]+(IG)^2)\right) $$
  We deduce that $\sigma$ induces natural transformations $\sigma_{|\A^{\otimes 2} }: \A^{\otimes 2} \to  \A^{\otimes 2}$ and $\overline{\sigma}: \A \to  \A$.  
  
  Since $([y]-[1])([y^{-1}]-[1])=-([y]-[1])-([y^{-1}]-[1])$ in $(IG)^2$, we have 
  $$([y^{-1}]-[1]+(IG)^2)\otimes ([x^{-1}]-[1]+(IG)^2) =(-([y]-[1])+(IG)^2)\otimes (-([x]-[1])+(IG)^2)$$
  $$=([y]-[1]+(IG)^2)\otimes ([x]-[1]+(IG)^2)$$
  giving the explicit description of $\sigma_{|\A^{\otimes 2} }$. 
  
For that of $\overline{\sigma}$, we have:
$$p \circ \sigma_G ([y]-[1]+(IG)^3)=p([y^{-1}]-[1]+(IG)^3)=[y^{-1}]-[1]+(IG)^2=-([y]-[1])+(IG)^2.$$
\end{proof}

The action of $\mathfrak{S}_2$ on $\mathcal{P}_2$ corresponds to an action of $\mathfrak{S}_2$ on $\mathcal{P}_2^{\#}$.

In order to describe the functor $\mathbf{A}_1(n,-)_{\mathbf{0} }$ in Proposition \ref{A11},
we need the following results on the second Passi functor.

\begin{prop} \label{P2_equ}
We have a natural equivalence in $\f_{\mathcal{L}ie}$:
$$\alpha^{-1}(\mathcal{P}_2^{\#})\simeq Cat\mathcal{L}ie(2,-)$$
hence $\alpha^{-1}(\mathcal{P}_2^{\#})$ is projective.
\end{prop}
\begin{proof}
The duality functor and $\alpha^{-1}$ being exact functors, we deduce from the non-split exact sequence (\ref{ses-Passi}), the following non-split exact sequence in $\mathcal{F}_{Lie}$:
\begin{equation} \label{sec-PP2}
\xymatrix{
 0 \ar[r] & \alpha^{-1}(\A^\#)  \ar[r] &\alpha^{-1}(\mathcal{P}_2^{\#})  \ar[r] &\alpha^{-1}((\A^{\otimes 2})^\# )\ar[r] & 0.
  }
  \end{equation}
  By (\ref{atomic}), $\alpha^{-1}(\A^\#)$ and $\alpha^{-1}((\A^{\otimes 2})^\# )$ are atomic functors given by: $\alpha^{-1}(\A^\#)=\mathbb{K}[1]$ and $\alpha^{-1}((\A^{\otimes 2})^\# )=\mathbb{K}[\mathfrak{S}_2][2]$. We deduce that $\alpha^{-1}(\mathcal{P}_2^{\#}) $ is non-zero only on $1$ and $2$ and that 
  $$\alpha^{-1}(\mathcal{P}_2^{\#})(1)= \mathbb{K} \quad \text{and} \quad \alpha^{-1}(\mathcal{P}_2^{\#})(2)= \mathbb{K}[\mathfrak{S}_2].$$
   
  The functor $Cat\mathcal{L}ie(2,-): Cat\mathcal{L}ie \to \ev$ is non-zero only on $1$ and $2$ and we have:
 
$$Cat\mathcal{L}ie(2,1) \simeq  \mathbb{K} \quad \text{and}\quad Cat\mathcal{L}ie(2,2) \simeq  \mathbb{K}[\mathfrak{S}_2].$$
By the Yoneda lemma:
$$\text{Hom}(Cat\mathcal{L}ie(2,-), \alpha^{-1}(\mathcal{P}_2^{\#})) \simeq \alpha^{-1}(\mathcal{P}_2^{\#})(2)= \mathbb{K}[\mathfrak{S}_2].$$
Let $\nu: Cat\mathcal{L}ie(2,-) \to \alpha^{-1}(\mathcal{P}_2^{\#})$ be the natural transformation corresponding to $[Id]$ by the previous isomorphism. By naturality of $\nu$, we have:
$$\alpha^{-1}(\mathcal{P}_2^{\#})(\mu_1^2) \circ \nu_2=\nu_1 \circ Cat\mathcal{L}ie(2,-)(\mu_1^2).$$

By construction, $\nu_2$ is an isomorphism and, since the short exact sequence (\ref{sec-PP2}) is non-split, $\alpha^{-1}(\mathcal{P}_2^{\#})(\mu_1^2)\not= 0$. We deduce that $\nu_1 \not= 0$ and so it is an isomorphism. Consequently, $\nu$ is a natural equivalence.
 
\end{proof}

\begin{cor}  \label{P2_inj}
The functor $\mathcal{P}_2^{\#}$ is projective in  the category of polynomial functors on $\gr^{op}$.
\end{cor}
\begin{proof}
The functor $\alpha$ is  an equivalence of categories and, by Proposition \ref{P2_equ}, $\alpha^{-1}(\mathcal{P}_2^{\#})$ is projective in $Cat\mathcal{L}ie$.
\end{proof}

\begin{cor} \label{cor-Passi}
For $n \in \nat$, we have a natural equivalence:
$$\alpha^{-1}(\mathcal{P}_2^{\#} \underset{\mathfrak{S}_2}{\otimes} \mathbb{K}[F_n]) \simeq  Cat\mathcal{L}ie(2,-) \underset{\mathfrak{S}_2}{\otimes} \mathbb{K}[F_n],$$
where the action of $\mathfrak{S}_2$ on $\mathbb{K}[F_n]$ is given by taking the inverse in $F_n$: $v\mapsto v^{-1}$ and the action of $\mathfrak{S}_2$ on $\mathcal{P}_2^{\#}$ is described in Lemma \ref{action-Passi}.
\end{cor}

%%%%%%%%%%%%%%%%

\section{Habiro-Massuyeau's category}

\subsection{Definition}  \label{rappels-HM} 
In \cite[Section 4.1]{HM}, Habiro and Massuyeau consider Jacobi diagrams on a $1$-manifold coloured by elements of a group (see also \cite{GL, ST}). In order to avoid the confusion with the fact that we will also consider Jacobi diagrams where the univalent vertices are "coloured" by a set, we replace the terminology used by Habiro and Massuyeau by \textit{beaded Jacobi diagrams} (following, for example, \cite{GR}).

For $d\geq 0$, let $X_d$ be the oriented $1$-manifold consisting of $d$ arc components. Recall that a \textit{Jacobi diagram} $D$ on $X_d$ is a uni-trivalent graph such that each trivalent vertex is oriented, the set of univalent vertices is embedded into the interior of $X_d$ and  each connected component of $D$ contains at least one univalent vertex. 
When a Jacobi diagram $D$ on $X_d$ is drawn in the plane, we draw the $1$-manifold $X_d$ with solid lines, the Jacobi diagram part $D$ with dashed lines and we assume counterclockwise orientation for the trivalent vertices of $D$.

For $G$ a group, a \textit{$G$-beaded Jacobi diagram on $X_d$} is a Jacobi diagram $D$ on $X_d$ whose graph edges are oriented and a $G$-valued function on a finite subset of $(\text{Int}(X_d) \cup D) \setminus \text{Vert}(D)$. This function labels the oriented edges of $D$ and the arcs of $X_d$, by elements in $G$.
In figures, the labels are encoded by "beads" coloured with elements of $G$.

Two  $G$-beaded Jacobi diagrams on $X_d$ are said to be \textit{equivalent} if they are related by the following sequence of local moves (see \cite[(4.1) and p.618]{HM}), where $w,x \in G$:

 \[
	\vcenter{\hbox{\begin{tikzpicture}[baseline=1.8ex,scale=0.7]
	\draw[-To,>=latex]  (0,0) to (2,0);
		\draw (2.1,0) node[right] {$\sim$};
	\draw[-To,>=latex]  (3,0) to(5,0);
		\draw[fill=white] (0.5,0) circle (3pt);
		\draw (0.5,0) node[above] {$\scriptstyle{w}$};
			\draw[fill=white] (1.5,0) circle (3pt);
		\draw (1.5,0) node[above] {$\scriptstyle{x}$};
			\draw[fill=white] (4,0) circle (3pt);
	\draw (4,0) node[above] {$\scriptstyle{wx}$};
	
		\draw (5.5,0) node[right] {;};	

	\draw[-To,>=latex]  (7,0) to (9,0);
		\draw (9.1,0) node[right] {$\sim$};
	\draw[-To,>=latex]  (10,0) to(12,0);
		\draw[fill=white] (8,0) circle (3pt);
		\draw (8,0) node[above] {$\scriptstyle{1}$};
		
			\draw (12.5,0) node[right] {;};	
			
			\draw[thick, dotted][-To,>=latex]  (0,-4) to (2,-4);
		\draw (2.1,-4) node[right] {$\sim$};
	\draw[thick, dotted][-To,>=latex]  (5,-4) to(3,-4);
		\draw[fill=white] (1,-4) circle (3pt);
		\draw (1,-4) node[above] {$\scriptstyle{w}$};
					\draw[fill=white] (4,-4) circle (3pt);
	\draw (4,-4) node[above] {$\scriptstyle{w^{-1}}$};

\draw[thick, dotted][-To,>=latex]  (0,-2) to (2,-2);
		\draw (2.1,-2) node[right] {$\sim$};
	\draw[thick, dotted][-To,>=latex]  (3,-2) to(5,-2);
		\draw[fill=white] (0.5,-2) circle (3pt);
		\draw (0.5,-2) node[above] {$\scriptstyle{w}$};
			\draw[fill=white] (1.5,-2) circle (3pt);
		\draw (1.5,-2) node[above] {$\scriptstyle{x}$};
			\draw[fill=white] (4,-2) circle (3pt);
	\draw (4,-2) node[above] {$\scriptstyle{wx}$};

			\draw (5.5,-2) node[right] {;};	

	\draw[thick, dotted][-To,>=latex]  (7,-2) to (9,-2);
		\draw (9.1,-2) node[right] {$\sim$};
	\draw[thick, dotted][-To,>=latex]  (10,-2) to(12,-2);
		\draw[fill=white] (8,-2) circle (3pt);
		\draw (8,-2) node[above] {$\scriptstyle{1}$};
		
			\draw (12.5,-2) node[right] {;};	
			
	\draw[-To,>=latex]  (14,0) to (16,0);
		\draw[thick, dotted][-To,>=latex]  (15,0) to (16,1);
		\draw (16.1,0) node[right] {$\sim$};
	\draw[-To,>=latex]  (17,0) to(19,0);
	\draw[thick, dotted][-To,>=latex]  (18,0) to (19,1);
		\draw[fill=white] (14.5,0) circle (3pt);
		\draw (14.5,0) node[above] {$\scriptstyle{w}$};
					\draw[fill=white] (18.5,0) circle (3pt);
		\draw (18.5,0) node[below] {$\scriptstyle{w}$};
			\draw[fill=white] (18.5,.5) circle (3pt);
		\draw (18.5,.5) node[above] {$\scriptstyle{w}$};
		
		\draw[thick, dotted][-To,>=latex]  (14,-2) to (15,-2);
			\draw[thick, dotted][-To,>=latex]  (15,-2) to (16,-2);
		\draw[thick, dotted][-To,>=latex]  (15,-2) to (16,-1);
		\draw (16.1,-2) node[right] {$\sim$};
	\draw[thick, dotted][-To,>=latex]  (17,-2) to(18,-2);
	\draw[thick, dotted][-To,>=latex]  (18,-2) to(19,-2);
	\draw[thick, dotted][-To,>=latex]  (18,-2) to (19,-1);
		\draw[fill=white] (14.5,-2) circle (3pt);
		\draw (14.5,-2) node[above] {$\scriptstyle{w}$};
					\draw[fill=white] (18.5,-2) circle (3pt);
		\draw (18.5,-2) node[below] {$\scriptstyle{w}$};
			\draw[fill=white] (18.5,-1.5) circle (3pt);
		\draw (18.5,-1.5) node[above] {$\scriptstyle{w}$};

		\end{tikzpicture}}}
		\]

For example, these two  $G$-beaded Jacobi diagrams on $X_2$, where $w_1, w_2, w_3 \in G$, are equivalent:

 \[
	\vcenter{\hbox{\begin{tikzpicture}[baseline=1.8ex,scale=0.5]
	
	\draw[->,>=latex]  (4,0) to[bend right=45] (0,0);
	\draw[->,>=latex]  (10,0) to[bend right=45] (6,0);
\draw[thick, dotted][-To,>=latex]   (5,2) to  (2.5,1.17);
				\draw[thick, dotted]  (2.5,1.17) to (2, 1);
			\draw[thick, dotted] [-To,>=latex]    (8,1) to (7, 1.33);
					\draw[thick, dotted] (7, 1.33) to (5, 2);
									\draw[thick, dotted]  (6.3,0.95) to (5, 2);
						\draw[thick, dotted][-To,>=latex]   (6.8,.6) to (6.3,0.95);

		\draw[fill=white] (.8,.6) circle (3pt);
		\draw (.8,.6) node[above] {$\scriptstyle{w_1}$};
		\draw[fill=white] (3.2,.6) circle (3pt);
		\draw (3.3,.6) node[below] {$\scriptstyle{w_3}$};
		\draw[fill=white] (3.1,1.3) circle (3pt);
		\draw (3.1,1.3) node[above] {$\scriptstyle{w_3}$};
		\draw[fill=white] (4,1.65) circle (3pt);
		\draw (4,1.7) node[above] {$\scriptstyle{w_2}$};
		\draw[fill=white] (9.2,.6) circle (3pt);
		\draw (9.2,.6) node[above] {$\scriptstyle{w_2}$};
			\draw[fill=white] (6.4,1.55) circle (3pt);
		\draw (8, 1.5) node[above] {$\scriptstyle{w_1(w_2)^{-1}}$};
	\draw[fill=white] (7.3,0.8) circle (3pt);
		\draw (7.6, 0.8) node[below] {$\scriptstyle{w_1}$};
		\draw[fill=black] (5,2) circle (1.5pt);
		\end{tikzpicture}}}
\quad \sim \quad 
\vcenter{\hbox{\begin{tikzpicture}[baseline=1.8ex,scale=0.5]
	
	\draw[->,>=latex]  (4,0) to[bend right=45] (0,0);
	\draw[->,>=latex]  (10,0) to[bend right=45] (6,0);
		\draw[thick, dotted][-To,>=latex]   (5,2) to  (2.5,1.17);
				\draw[thick, dotted]  (2.5,1.17) to (2, 1);
			\draw[thick, dotted] [-To,>=latex]    (8,1) to (7, 1.33);
					\draw[thick, dotted] (7, 1.33) to (5, 2);
						\draw[thick, dotted]  (6.3,0.95) to (5, 2);
						\draw[thick, dotted][-To,>=latex]   (6.8,.6) to (6.3,0.95);
		\draw[fill=white] (.8,.6) circle (3pt);
		\draw (.8,.6) node[above] {$\scriptstyle{w_1w_3}$};
			\draw[fill=white] (4,1.65) circle (3pt);
		\draw (4,1.7) node[above] {$\scriptstyle{w_2w_1}$};
		\draw[fill=white] (5.5,1.5) circle (3pt);
		\draw (5.5, 1.5) node[below] {$\scriptstyle{w_1}^{-1}$};
	\draw[fill=white] (7.3,0.8) circle (3pt);
		\draw (7.6, 0.8) node[below] {$\scriptstyle{w_1w_2}$};
		\draw[fill=black] (5,2) circle (1.5pt);

		\end{tikzpicture}}}
		\]

In particular, each $G$-beaded Jacobi diagram on $X_d$ is equivalent to a $G$-beaded Jacobi diagram of the form:
\[
	\vcenter{\hbox{\begin{tikzpicture}[baseline=1.8ex,scale=0.4]
	
	\draw[->,>=latex]  (4,0) to[bend right=45] (0,0);
	\draw[->,>=latex]  (10,0) to[bend right=45] (6,0);
	\draw[->,>=latex]  (18,0) to[bend right=45] (14,0);
		\draw[thick, dotted][->,>=latex]   (1.5,0.7) to (1.5, 4);
		\draw (2,2) node[below] {$\scriptstyle{\ldots}$};
		\draw[thick, dotted][->,>=latex]   (2.5,0.7) to (2.5, 4);
		\draw[thick, dotted][->,>=latex]   (7.5,0.7) to (7.5, 4);
		\draw (8,2) node[below] {$\scriptstyle{\ldots}$};
		\draw[thick, dotted][->,>=latex]   (8.5,0.7) to (8.5, 4);
		\draw[thick, dotted][->,>=latex]   (15.5,0.7) to (15.5, 4);
		\draw (16,2) node[below] {$\scriptstyle{\ldots}$};
		\draw[thick, dotted][->,>=latex]   (16.5,0.7) to (16.5, 4);
	\draw (2,0) node[below] {$\scriptstyle{1}$};
	\draw (8,0) node[below] {$\scriptstyle{2}$};
	\draw (12,0) node[below] {$\scriptstyle{\ldots}$};
	\draw (16,0) node[below] {$\scriptstyle{d}$};	
		\draw[fill=white] (.8,.6) circle (4pt);
		\draw (.8,.6) node[above] {$\scriptstyle{w_1}$};
		\draw[fill=white] (6.8,.6) circle (4pt);
		\draw (6.8,.6) node[above] {$\scriptstyle{w_2}$};
		\draw[fill=white] (14.8,.6) circle (4pt);
		\draw (14.8,.6) node[above] {$\scriptstyle{w_{d}}$};
			\end{tikzpicture}}}
		\]
where $w_1, \ldots, w_{d} \in G$ and where we can have beads on the Jacobi diagram represented by the dashed part in the figure.

In \cite[Section 4.2]{HM} Habiro and Massuyeau define the linear category $\mathbf{A}$ of Jacobi diagrams in handlebodies. This category has $\mathbb{N}$ as objects and for $n,m \in \mathbb{N}$, $\mathbf{A}(n,m)$ is the vector space generated by the equivalence classes of $F_n$-beaded Jacobi diagrams on $X_m$ modulo the STU relation. The composition in the category $\mathbf{A}$ is quite complicated and we refer the reader to \cite[Section 4.2]{HM} for its definition.

By \cite[Section 4.3]{HM}, the linear category $\mathbf{A}$ admits a symmetric monoidal structure given on objects by the addition of integers. We denote this monoidal structure by $\odot$.

\subsection{Two gradings and sub-(semi)-categories} 

In \cite[Section 4.4]{HM} the authors define two gradings  on the morphisms of  $\mathbf{A}$. 
The first one is a $\mathbb{N}$-grading given by the degree of the Jacobi diagram: for $m, n \in \mathbb{N}$, $\mathbf{A}(n,m)$ can be decomposed as a direct sum with respect to the degree $d$ of the Jacobi diagrams
\begin{equation}\label{decomposition-degree}
\mathbf{A}(n,m) \simeq \underset{d \in \mathbb{N}}{\bigoplus} \mathbf{A}_d(n,m).
\end{equation}
This grading is compatible with the composition in the category $\mathbf{A}$ giving maps 
\begin{equation}\label{degree-composition}
\circ: \mathbf{A}_{d'}(m,n') \times \mathbf{A}_d(n,m)\to \mathbf{A}_{d+d'}(n,n') 
\end{equation}
for $d,d' \in \nat$.

The second grading is  a $\gr^{op}$-grading: the \textit{homotopy class} of an $(m,n)$-Jacobi diagram $D$ on $X_n$ is the homomorphism $h(D): F_n \to F_m$ that maps each generator $x_j$ to the product of the beads along the $j^{th}$-oriented component of $X_n$. We have
\begin{equation}\label{decomposition-homotopy}
\mathbf{A}(m,n)=\underset{f \in \gr^{op}(m,n)}{\bigoplus} \mathbf{A}(m,n)_f.
\end{equation}
Note that the identity morphism in $\mathbf{A}(n,n)$ is in the homotopy class of the identity homomorphism $F_n \to F_n$. 
This grading is compatible with the composition in the category $\mathbf{A}$.
\begin{equation}\label{compo-grad-2}
\circ:  \mathbf{A}(m,n')_{g} \times \mathbf{A}(n,m)_f \to \mathbf{A}(n,n')_{g\circ f}
\end{equation}
for $f \in \gr^{op}(n,m)=\gr(m,n)$ and $g \in \gr^{op}(m,n')=\gr(n',m)$.

Using these gradings we can consider the following subcategory and sub-semicategory of $\mathbf{A}$. Recall that a semicategory is defined as a category without the condition on the existence of identity morphisms (see \cite[Section 4]{Mitchell}).

Taking degree $d=0$, by (\ref{degree-composition}) we have maps
$$\circ:  \mathbf{A}_{0}(m,n')  \times  \mathbf{A}_0(n,m)  \to \mathbf{A}_{0}(n,n').$$

Hence $\mathbf{A}$ has a subcategory, denoted by $\mathbf{A}_0$, such that $Obj(\mathbf{A}_0)=Obj(\mathbf{A})$ and the morphisms in $\mathbf{A}_0$ are given by Jacobi diagrams of degree $0$. By \cite[p. 630]{HM} we have an isomorphism of linear categories
\begin{equation}\label{eq-cat}
h: \mathbf{A}_0 \xrightarrow{\simeq} \mathbb{K}\gr^{op}.
\end{equation}
This isomorphism comes from the fact that $\mathbf{A}_0(n,m)$ is generated by $F_n$-beaded empty Jacobi diagrams on $X_m$. So we have only beads on the arcs of $X_m$. Such a choice of beads corresponds to a homomorphism: $F_n \to F_m$ sending $x_i$ to the bead on the $i$-th arc of $X_m$.

Via the isomorphism $h$ given in (\ref{eq-cat}), the generators $(m_1,m_2,m_3,m_4,m_5)$ of $\gr$ recalled in Section \ref{gr-op} correspond to the morphisms $(\eta, \mu, \epsilon, S, \Delta)$ given in \cite[(5.28)]{HM}.

Recall that $\mathbf{0} \in \gr^{op}(n,m)=\gr(m,n)$ is the composition $m \to 0 \to n$ in $\gr$. By Section \ref{rappels-HM}, a $F_n$-beaded Jacobi diagram $D$ on $X_m$ in the homotopy class of $\mathbf{0}$, is represented by a $F_n$-beaded Jacobi diagram without beads on $X_m$ (but there may be beads on $D$). 

By (\ref{compo-grad-2}) we have maps
$$\circ: \mathbf{A}(m,n')_{\mathbf{0}} \times \mathbf{A}(n,m)_{\mathbf{0}} \to \mathbf{A}(n,n')_{\mathbf{0}}.$$

We deduce that $\mathbf{A}$ has a sub-semicategory, denoted by $\mathbf{A}(-,-)_{\mathbf{0}}$, such that $Obj(\mathbf{A}(-,-)_{\mathbf{0}})=Obj(\mathbf{A})$ and the morphisms in $\mathbf{A}(-,-)_{\mathbf{0}}$ are given by beaded Jacobi diagrams in the homotopy class of $\mathbf{0}$. 

%%%%%%%%%%%%%%%

\section{Projective generators on Habiro-Massuyeau's category}
Let $\mathbf{A}\text{-Mod}$ be the category of $\mathbb{K}$-linear  functors from $\mathbf{A}$ to the category $\ev$.
For $n \in \mathbb{N}$, $\mathbf{A}(n,-):\mathbf{A}\to\ev$ is a linear functor. By the enriched Yoneda lemma, for $F:\mathbf{A}\to \ev$ a $\mathbb{K}$-linear functor, we have an isomorphism:
$$\text{Hom}_{\mathbf{A}\text{-Mod}}(\mathbf{A}(n,-), F) \simeq F(n).$$
We deduce that the functors $\mathbf{A}(n,-)$, for $n \geq 0$, are projective generators of $\mathbf{A}\text{-Mod}$.

Note that, for $d \in \mathbb{N}$,  $\mathbf{A}_d(n,-)$ do not define subfunctors of $\mathbf{A}(n,-)$  since the degree of Jacobi diagrams is not preserved by composition. 
However, for $n \in \mathbb{N}$, by restriction to $\mathbf{A}_0$, we have linear functors $\mathbf{A}(n,-): \mathbf{A}_0 {\simeq} \mathbb{K}\gr^{op} \to \ev$ and so functors
$$\mathbf{A}(n,-): \gr^{op} \to \ev.$$
Then, by (\ref{degree-composition}), the grading  by the degree of the Jacobi diagrams defines subfunctors 
$$\mathbf{A}_d(n,-): \gr^{op} \to \ev.$$

\begin{rem}
In \cite{Katada, KatadaII}, Katada studies the functor $\mathbf{A}(0,-): \gr^{op} \to \ev$ and, for $d\in \nat$, its subfunctors $\mathbf{A}_d(0,-)$ denoted by $A_d$ in  \cite{Katada, KatadaII}. In \cite{Katada}, she proves that $\mathbf{A}_d(0,-)$ is a polynomial functor of degree $2d$ which is an outer functor. She also gives a complete description of the functor $\mathbf{A}_1(0,-)$ and the more complicated case of the functor $\mathbf{A}_2(0,-)$. In \cite[Theorem 10.1]{KatadaII}, she gives a direct decomposition of the functor $\mathbf{A}_d(0,-)$ for $d\geq 1$ (see also Proposition \ref{direct-decomposition} for another proof) and obtains in \cite[Proposition 10.2]{KatadaII} that this is an indecomposable decomposition.
\end{rem}

\begin{rem}
More generally, $\mathbf{A}(-,-):\mathbf{A}^{op} \times \mathbf{A}\to\ev$ is a linear functor and, by restriction, we have a functor $\mathbf{A}(-,-):\gr \times \gr^{op} \to\ev$. The study of these bifunctors will be done elsewhere.
\end{rem}
%%%%%%%%%%%%%%%%%%%%%%
\subsection{Generalities on the functors $\mathbf{A}(n, -)$ and $\mathbf{A}_d(n, -)$}

The first result of this section shows that the functors $\mathbf{A}(n, -)$  are connected to each other by injective natural transformations.
Let $\epsilon \in \mathbf{A}_0(1,0)=\mathbf{A}(1,0)$ be the morphism corresponding, via the isomorphism $h$ of (\ref{eq-cat}) to the morphism $m_3$ given in Section  \ref{gr-op}. For $n\geq 1$ we have $\mathbf{A}(n,0) \simeq \mathbb{K}[\epsilon^{\odot n}]\simeq \mathbb{K}$ and $\mathbf{A}(0,0)=\mathbf{A}_0(0,0)=\mathbb{K}$. So $0$ is a terminal object in the $\mathbb{K}$-linear category $\mathbf{A}$. We deduce that, for $n\geq 1$, the functors $\mathbf{A}(n,-): \gr^{op} \to \ev$ are not reduced. We denote by $\overline{\mathbf{A}(n,-)}$ the reduced subfunctor of $\mathbf{A}(n,-)$. In particular we have: $\mathbf{A}(n,-) \simeq \mathbb{K} \oplus \overline{\mathbf{A}(n,-)}$. Note that $0$ is far from being an initial object in $\mathbf{A}$ and that $\mathbf{A}(0,-)$ is reduced.
\begin{lem} \label{subfunctor}
For $d,n \in \mathbb{N}$, the precomposition with $Id_n \odot \epsilon \in \mathbf{A}_0(n+1,n)$  gives injective natural transformations
$$ \mathbf{A}(n, -) \hookrightarrow  \mathbf{A}(n+1,-); \qquad \mathbf{A}_d(n, -) \hookrightarrow  \mathbf{A}_d(n+1,-).$$

\end{lem}
\begin{proof}
 For $n=0$, the injectivity follows from  \cite[Lemma 4.5]{HM} and the general case is a consequence of the generalization of this Lemma given in \cite[Remark 4.6]{HM}. By (\ref{degree-composition}) the composition preserves the degree of the Jacobi diagram.
\end{proof}
 
In  \cite[Proposition 8.1]{Katada}, Katada proves that the functor $\mathbf{A}_d(0, -)$ is polynomial of degree $2d$.
The following Proposition shows that the polynomiality of the functors $\mathbf{A}_d(n, -)$ is an infrequent phenomenon.
\begin{prop} \label{poly-A_d(n)}
For $d,n \in \mathbb{N}$, the functor $\mathbf{A}_d(n, -):\gr^{op} \to \ev$ is polynomial iff $n=0$. 
\end{prop}

The proof of this proposition is based on the following lemmas.
\begin{lem} \label{lm-d=0}
For $n \in \mathbb{N}$, we have an isomorphism of functors:
$\mathbf{A}_0(n,-)\simeq P_n.$ In particular, $\mathbf{A}_0(0,-)\simeq \mathbb{K}$ is polynomial of degree $0$ and, for $n\geq 1$, $\mathbf{A}_0(n,-)$ is neither polynomial nor analytic.
\end{lem}
\begin{proof}
This follows from the isomorphism $h: \mathbf{A}_0 \xrightarrow{\simeq} \mathbb{K}\gr^{op}$ given in (\ref{eq-cat}).
\end{proof}
For $n\geq 1$, note that $\mathbf{A}_0(0,n)=\mathbb{K}[\eta^{\odot n}]$ where $\eta \in \mathbf{A}_0(0,1)$ is the morphism defined in \cite[Section 5.6, (5.28)]{HM} corresponding to $m_1$ (see Section \ref{gr-op}) via the isomorphism $h$ of (\ref{eq-cat}) .

 \begin{lem}\label{lm-d>0}
 For $d \geq 1$ and $n\geq 1$ the functor $\mathbf{A}_d(n, -):\gr^{op} \to \ev$ is not polynomial.
 \end{lem} 
 \begin{proof}
 We will prove that for $k\geq 2d+1$, $\widetilde{cr}_{k}(\mathbf{A}_d(n,-)) \not=0$.

By Section \ref{rappels-poly}, $\widetilde{cr}_{k}(\mathbf{A}_d(n,-))(1,\ldots, 1)$ is the cokernel of the following homomorphism:
$$\overset{k}{\underset{l=1}{\bigoplus}}\mathbf{A}_d(n,-)(F_{k-1}) \xrightarrow{(\mathbf{A}_d(n,-)(r^{k}_{\hat{1}}), \ldots, \mathbf{A}_d(n,-)(r^{k}_{\hat{k}}))} \mathbf{A}_d(n,-)(F_{k}) $$
 $\mathbf{A}_d(n,-)(F_{k-1}) \not= 0$ and a generator of $\mathbf{A}_d(n,-)(F_{k-1})$ is represented by a $F_n$-beaded Jacobi diagram $D$ on $X_{k-1}$ having $2d$ vertices. For $1\leq i\leq k$,
$\mathbf{A}_d(n,-)(r^{k}_{\hat{i}})(D)$ is the $F_n$-beaded Jacobi diagram $D$ on $X_{k}$ obtained from $D$ by inserting the $F_n$-beaded arc:
\[
	\vcenter{\hbox{\begin{tikzpicture}[baseline=1.8ex,scale=0.4]
	\draw[->,>=latex]  (-2,0) to[bend right=55] (-4,0);
			\draw[fill=white] (-3,0.5) circle (4pt);
			\draw (-3,0.5) node[above] {$\scriptstyle{1}$};
		\end{tikzpicture}}}
		\]
 between the $(i-1)$-th and the $i$-th arc of $D$. So the following $F_n$-beaded Jacobi diagram on $X_k$ is a non-zero element  in $\widetilde{cr}_{k}(\mathbf{A}_d(n,-))(1,\ldots, 1)$:

 \[
	\vcenter{\hbox{\begin{tikzpicture}[baseline=1.8ex,scale=0.4]
	\draw[->,>=latex]  (-2,0) to[bend right=45] (-4,0);
	\draw[->,>=latex]  (2,0) to[bend right=45] (0,0);
	\draw[->,>=latex]  (8,0) to[bend right=45] (6,0);
	\draw[->,>=latex]  (12,0) to[bend right=45] (10,0);
	\draw[->,>=latex]  (18,0) to[bend right=45] (16,0);
	\draw[->,>=latex][thick, dotted]  (1,0.5) to[bend right=45] (-3,0.5);
	\draw[->,>=latex][thick, dotted]  (11,0.5) to[bend right=45] (7,0.5);
\draw (-3,0) node[below] {$\scriptstyle{1}$};
	\draw (1,0) node[below] {$\scriptstyle{2}$};
	\draw (4,0) node[below] {$\scriptstyle{\ldots}$};
	\draw (7,0) node[below] {$\scriptstyle{2d-1}$};
	\draw (11,0) node[below] {$\scriptstyle{2d}$};
		\draw (14,0) node[below] {$\scriptstyle{\ldots}$};
		\draw (17,0) node[below] {$\scriptstyle{k}$};
		\draw[fill=white] (-1,1.3) circle (4pt);
		\draw[fill=white] (9,1.3) circle (4pt);
		\draw[fill=white] (17,.5) circle (4pt);
		\draw (17,0.5) node[above] {$1\not= w \in F_n$};
	\end{tikzpicture}}}
		\]

 \end{proof}
 
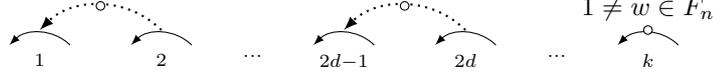
  \begin{proof}[Proof of Proposition \ref{poly-A_d(n)}]
  The case $d=0$ follows from Lemma \ref{lm-d=0}.   For $d \geq 1$, if $n\geq 1$ the functor $\mathbf{A}_d(n, -)$ is not polynomial by Lemma \ref{lm-d>0}. The polynomiality of $\mathbf{A}_d(0, -)$ is given by  \cite[Proposition 8.1]{Katada} (see also Corollary  \ref{poly-Katada}).
 
 \end{proof}
 
%%%%%%%
\subsection{Filtration of the functors $\mathbf{A}(n, -)$ and $\mathbf{A}_d(n, -)$}

For $n,m, t \in \nat$, let $\mathbf{A}^t(n,m)$ be the subspace of $\mathbf{A}(n, m)$ generated by Jacobi diagrams having at least $t$ trivalent vertices. Similarly we define $\mathbf{A}^t_{d}(n, m)$. We have the following result:
\begin{prop} \label{filtration}
For $d,m \in \nat$, the functors $\mathbf{A}(n, -)$ and $\mathbf{A}_d(n, -)$ have a filtration given by the subfunctors:
$$\mathbf{A}^t(n, -) \subset \mathbf{A}(n, -); \qquad \mathbf{A}^t_{d}(n, -) \subset \mathbf{A}_d(n, -).$$
\end{prop}
\begin{proof}
Let $D$ be a generator in $\mathbf{A}^t(n, m)$ and $f\in \gr^{op}(m,m')$. Via the isomorphism $\mathbb{K}[\gr^{op}] \simeq \mathbf{A}_0$, $f$ corresponds to an element in $\mathbf{A}_0(m,m')$. The composition in $\mathbf{A}$ is given by a suitable cabling of the Jacobi diagram of $D$ on the arcs of $X_{n'}$. This operations does not change the number of trivalent vertices in the Jacobi diagram.
\end{proof}

In \cite{Katada}, Katada considers the filtration 
$$0=\mathbf{A}^{2d-1}_{d}(0, -)\subset \ldots \subset \mathbf{A}^{1}_d(0, -)\subset \mathbf{A}^{0}_d(0, -)= \mathbf{A}_d(0, -).$$

%%%%%%%%%%%%%%%%
\section{The functors $\mathbf{A}(n,-)_{\mathbf{0}}$ and beaded open Jacobi diagrams}
For $d \in \mathbb{N}$ and $n\geq 1$, by Proposition \ref{poly-A_d(n)}, $\mathbf{A}_d(n,-)$ is \textit{not} polynomial.
In this section we introduce a subfunctor of $\mathbf{A}_d(n,-)$, which is polynomial and which coincides, for $n=0$, with $\mathbf{A}_d(0,-)$.

\subsection{Definition of the functors $\mathbf{A}(n,-)_{\mathbf{0}}$}
The functors $\mathbf{A}(n,-)_{\mathbf{0}}$ are defined using the $\gr^{op}$-grading of $\mathbf{A}$ which is compatible with the composition in $\mathbf{A}$ by (\ref{compo-grad-2}).

  We deduce from (\ref{compo-grad-2}) the following Proposition:
\begin{prop} \label{sous-foncteur}
For $n \in \mathbb{N}$, the $\gr^{op}$-grading gives rise to the subfunctor
$\mathbf{A}(n,-)_{\mathbf{0}}: \gr^{op} \to \ev$ of $\mathbf{A}(n,-): \gr^{op} \to \ev$ and the subfunctor $\mathbf{A}_d(n,-)_{\mathbf{0}}: \gr^{op} \to \ev$ of $\mathbf{A}_d(n,-): \gr^{op} \to \ev$.
\end{prop}
\begin{proof}
For $g \in \gr^{op}(m,n')$ and $h \in \gr^{op}(m,n')$ we have
$$\mathbf{0} \circ g=\mathbf{0} \quad \text{and} \quad h \circ \mathbf{0}=\mathbf{0}.$$
where $\mathbf{0} \in \gr^{op}(n,m)=\gr(m,n)$ is the homomorphism $F_m \to F_n$ sending each generator to $1$.  

\end{proof}
\begin{rem}
For $n,m \in \mathbb{N}$, the generators of $\mathbf{A}(n,m)_{\mathbf{0}}$ are those of $\mathbf{A}(n,m)$ which can be represented by an $F_n$-beaded Jacobi diagrams on $X_m$, without beads on $X_m$. 
\end{rem}

\begin{cor} \label{subfunctor-1}
For $d,n \in \mathbb{N}$, the precomposition with $Id_n \odot \epsilon \in \mathbf{A}_0(n+1,n)$  gives injective natural transformations
$$ \mathbf{A}(n, -)_{\mathbf{0}} \hookrightarrow  \mathbf{A}(n+1,-)_{\mathbf{0}}; \qquad \mathbf{A}_d(n, -)_{\mathbf{0}} \hookrightarrow  \mathbf{A}_d(n+1,-)_{\mathbf{0}}.$$
\end{cor}

Note that, for $n=0$, we have $\mathbf{A}(0,-)_{\mathbf{0}}=\mathbf{A}(0,-)$ and $\mathbf{A}_d(0,-)_{\mathbf{0}}=\mathbf{A}_d(0,-)$.

Similarly to Proposition \ref{filtration} we have the following result:
\begin{prop} \label{filtration2}
For $d,n \in \nat$, the functors $\mathbf{A}(n, -)_{\mathbf{0}}$ and $\mathbf{A}_d(n, -)_{\mathbf{0}}$ have a filtration given by the subfunctors:
$$\mathbf{A}^t(n, -)_{\mathbf{0}} \subset \mathbf{A}(n, -)_{\mathbf{0}}; \qquad \mathbf{A}^t_{d}(n, -)_{\mathbf{0}} \subset \mathbf{A}_d(n, -)_{\mathbf{0}}.$$
\end{prop}
\begin{rem} \label{A_0}
We have $\mathbf{A}_0(n, -)_{\mathbf{0}}=\mathbb{K}$ and $\mathbf{A}_0(n,-)\simeq P_n \simeq \mathbf{A}_0(n, -)_{\mathbf{0}} \oplus \overline{P}_n.$
\end{rem}

%%%%%%%%%
\subsection{The $Cat\mathcal{L}ie$-modules $J^{F_m}$ of $F_m$-beaded  open Jacobi diagrams} \label{JFm-section}

Recall that an \textit{open Jacobi diagram} is a uni-trivalent graph such that each trivalent vertex is oriented and having at least one univalent vertex in each connected component.
For generalities on open Jacobi diagrams we refer the reader to \cite[Section 5.6]{CDM}.

For $Z$ a set, a $Z$-\textit{labelled} open Jacobi diagram is an open Jacobi diagram  $D$ and a bijection: $\{\text{univalent vertices of } D\} \xrightarrow{\simeq } Z$.
Note that in \cite[p.13]{Katada}, $Z$-labelled open Jacobi diagrams are called \textit{special $Z$-coloured} open Jacobi diagrams.

For $G$ a group, a \textit{$G$-beaded open Jacobi diagram} is an open Jacobi diagram  whose graph edges are oriented and a map from a finite subset of $ D \setminus Vert(D)$ to $G$ which labels oriented edges of $D$ by elements in $G$.
In figures, the labels are encoded by "beads" coloured with elements of $G$.

Two  $G$-beaded open Jacobi diagrams are said to be \textit{equivalent} if they are related by the following local moves where $w,x \in G$:

 \begin{equation}\label{relations-ouvert}
	\vcenter{\hbox{\begin{tikzpicture}[baseline=1.8ex,scale=0.7]
	\draw[thick, dotted][-To,>=latex]  (0,0) to (2,0);
		\draw (2.1,0) node[right] {$\sim$};
	\draw[thick, dotted][-To,>=latex]  (3,0) to(5,0);
		\draw[fill=white] (0.5,0) circle (3pt);
		\draw (0.5,0) node[above] {$\scriptstyle{w}$};
			\draw[fill=white] (1.5,0) circle (3pt);
		\draw(1.5,0) node[above] {$\scriptstyle{x}$};
			\draw[fill=white] (4,0) circle (3pt);
	\draw (4,0) node[above] {$\scriptstyle{wx}$};
	
		\draw[thick, dotted] (5.5,0) node[right] {;};	
		
		\draw[thick, dotted][-To,>=latex]  (6.5,0) to (7.5,0);
		\draw(7.6,0) node[right] {$\sim$};
	\draw[thick, dotted][-To,>=latex]  (8.5,0) to(9.5,0);
		\draw[fill=white] (7,0) circle (3pt);
		\draw (7,0) node[above] {$\scriptstyle{1}$};

			\draw[thick, dotted] (10,0) node[right] {;};	
						
	\draw[thick, dotted][-To,>=latex]  (11,0) to (12,0);
	\draw[thick, dotted][-To,>=latex]  (12,0) to (13,0);
		\draw[thick, dotted][-To,>=latex]  (12,0) to (13,1);
		\draw[thick, dotted] (13.1,0) node[right] {$\sim$};
	\draw[thick, dotted][-To,>=latex]  (14,0) to(15,0);
		\draw[thick, dotted][-To,>=latex]  (15,0) to(16,0);
	\draw[thick, dotted][-To,>=latex]  (15,0) to (16,1);
		\draw[fill=white] (11.5,0) circle (3pt);
		\draw (11.5,0) node[above] {$\scriptstyle{w}$};
					\draw[fill=white] (15.5,0) circle (3pt);
		\draw (15.5,0) node[below] {$\scriptstyle{w}$};
			\draw[fill=white] (15.5,.5) circle (3pt);
		\draw (15.5,.5) node[above] {$\scriptstyle{w}$};
		
		\draw[thick, dotted] (16.5,0) node[right] {;};

			\draw[thick, dotted][-To,>=latex]  (17.5,0) to (18.5,0);
		\draw (18.6,0) node[right] {$\sim$};
	\draw[thick, dotted][-To,>=latex]  (20.5,0) to(19.5,0);
		\draw[fill=white] (18,0) circle (3pt);
		\draw (18,0) node[above] {$\scriptstyle{w}$};
					\draw[fill=white] (20,0) circle (3pt);
	\draw (20,0) node[above] {$\scriptstyle{w^{-1}}$};

			\end{tikzpicture}}}
		\end{equation}

For $G$ a group and $Z$ a set, $J^G(Z)$ is the quotient by the AS and the IHX relations, of the $\mathbb{K}$-vector space generated by equivalence classes of $Z$-labelled, $G$-beaded, open Jacobi diagrams. Let $J^G_d(Z)$ be the subspace of $J^G(Z)$ generated by the Jacobi diagram having $2d$ vertices.

A generator in $Cat\mathcal{L}ie(n,m)$
can be viewed  as a $\mathbf{m+n}$-labelled, $F_0=\{1\}$-beaded open Jacobi diagram. In this case the orientation of the edges can be taken arbitrarily (by the last  relation given in (\ref{relations-ouvert})). 

\begin{prop} \label{Foncteurs-J_d}
For $d \in \nat$, 
$\mathbf{n} \mapsto J^{F_m}_d(\mathbf{n} )$ has the structure of a $\mathbb{K}$-linear functor on $Cat\mathcal{L}ie$
\end{prop}
\begin{proof}
By the description of the category $Cat\mathcal{L}ie$ given in Section \ref{CatLie}, it is sufficient to define $J^{F_m}_d$ on the generators $\sigma \in Cat\mathcal{L}ie(n,n)\simeq \mathbb{K}[\mathfrak{S}_n]$ and $\mu_i^{n}\in Cat\mathcal{L}ie(n,n-1)$.

Let $D$ be a generator in $J^{F_m}_d(n )$, that is $D$ is represented by a $\mathbf{n}$-labelled, $F_m$-beaded, open Jacobi diagram.

The action of $Cat\mathcal{L}ie(n,n)\simeq \mathbb{K}[\mathfrak{S}_n]$ on $D$ is given by the permutation of the labels of univalent vertices.

To define  $J^{F_m}_d(\mu_i^{n})(D)$, consider  the open Jacobi diagram $D'$ obtained from $D$ by gluing the tree ${\hbox{\begin{tikzpicture}[baseline=1.8ex,scale=0.4]
	\draw[thick, dotted]  (-3,3) to (-2,2);
	\draw[thick, dotted] (-1,3) to (-2,2);
	\draw[thick, dotted]  (-2,1) to (-2,2);
	\draw (-3,3) node[above] {$\scriptstyle{i}$};
	\draw (-1,3) node[above] {$\scriptstyle{n}$};
		\end{tikzpicture}}}$ 
		to the corresponding univalent vertices of $D$.  Edges of $D'$ inherit  an orientation from $D$ and a labelling in $F_m$. Colouring  the univalent vertex of $D'$ without label by $i$, we obtain a $(\mathbf{n-1})$-labelled, $F_m$-beaded, open Jacobi diagram. The antisymmetry and Jacobi relations in the operad $\mathcal{L}ie$ imply that this construction is well-defined on $J^{F_m}_d(\mathbf{n} )$.
		
\end{proof}

%%%%%%%%
\subsection{The correspondence between $\mathbf{A}_d(n,-)_{\mathbf{0} }$ and  $J_d^{F_n}$} \label{main-thm}

We have the following Theorem:

\begin{thm} \label{thm-J_d}
For $n,d \in \mathbb{N}$, we have an equivalence of functors in $\mathcal{F}_{\omega}(\gr^{op}; \mathbb{K})$:
$$\alpha(J_d^{F_n}) \simeq \mathbf{A}_d(n,-)_{\mathbf{0} }.$$
\end{thm}

\begin{proof}
By definition, $Cat\mathcal{A}ss^u \underset{Cat\mathcal{L}ie}{\otimes } J_d^{F_n} $ is the coequalizer of the following diagram:
$$\xymatrix{
Cat\mathcal{A}ss^u \underset{\Sigma}{\otimes}  Cat\mathcal{L}ie \underset{\Sigma}{\otimes} J_d^{F_n}   \ar@<+0.5ex>[r]^-{L} \ar@<-0.5ex>[r]_-{R} &  Cat\mathcal{A}ss^u  \underset{\Sigma}{\otimes}   J_d^{F_n} 
}$$
where $L$ is defined using the functor $ Cat\mathcal{A}ss^u : (Cat\mathcal{L}ie )^{op}\to Func(\mathbb{K}[\gr^{op}]; \ev)$ and $R$ using the functor $J_d^{F_n} : Cat\mathcal{L}ie \to \ev$. More explicitly, $Cat\mathcal{A}ss^u \underset{Cat\mathcal{L}ie}{\otimes } J_d^{F_n} $ is the coequalizer of the following diagram:
$$\xymatrix{
\underset{k,i \in \nat}{\bigoplus} Cat\mathcal{A}ss^u(i,-) \underset{\mathfrak{S}_i}{\otimes}  Cat\mathcal{L}ie(k,i) \underset{\mathfrak{S}_k}{\otimes} J_d^{F_n}(k)   \ar@<+0.5ex>[r]^-{L} \ar@<-0.5ex>[r]_-{R} &  \underset{c \in \nat}{\bigoplus} Cat\mathcal{A}ss^u  (c,-) \underset{\mathfrak{S}_c}{\otimes}   J_d^{F_n} (c)
}$$
where $L$ is defined using the map $Cat\mathcal{A}ss^u(i,-) \underset{\mathfrak{S}_i}{\otimes}  Cat\mathcal{L}ie(k,i) \to Cat\mathcal{A}ss^u  (k,-)$  obtained using the functor $Cat\mathcal{L}ie \to Cat\mathcal{A}ss^u$  and $R$ is defined using the map $Cat\mathcal{L}ie(k,i) \underset{\mathfrak{S}_k}{\otimes} J_d^{F_n}(k) \to  J_d^{F_n}(i)$. 

Let $\mathcal{J}_d^{F_n}(c)$ be the set of $c$-labelled, $F_n$-beaded, open Jacobi diagrams.
For $l \in \nat$, we define a linear map
$${\underset{c \in \nat}{\bigoplus}} Cat\mathcal{A}ss^u  (c,l) \underset{\mathfrak{S}_c}{\otimes} \mathbb{K}[\mathcal{J}_d^{F_n}(c)] \xrightarrow{f_l} \mathbf{A}_d(n,-)_{\mathbf{0} }(l)$$
as follows: for $[\tilde{\alpha}]$ a generator  of $Cat\mathcal{A}ss^u  (c,l) $ represented by a set map $\alpha:c\to l$ and a given order on each of its fiber and $D$ a $c$-labelled, $F_n$-beaded, open Jacobi diagram, we define $f_l\left([\tilde{\alpha}] \otimes [D] \right)$
as being the Jacobi diagram on $X_l$ obtained by gluing the univalent vertices of $D$ labeled by the elements of $\alpha^{-1}(k)$ on the $k$-th component of $X_l$, respecting the order given on the fiber $\alpha^{-1}(k)$, for $1 \leq k\leq l$.

The map $f_l$ is well-defined with respect to the AS and IHX relations and so defines a linear map:
$${\underset{c \in \nat}{\bigoplus}} Cat\mathcal{A}ss^u  (c,l) \underset{\mathfrak{S}_c}{\otimes}{J}_d^{F_n}(c) \xrightarrow{f_l} \mathbf{A}_d(n,-)_{\mathbf{0} }(l)$$
which is compatible with the action of the symmetric group $\mathfrak{S}_l$.

By Section \ref{gr-op}, the PROP $\gr$ is generated by the permutations and the homomorphisms $m_i$ for $i\in \{1,2,3,4,5\}$. To prove that the linear maps $f_l$ define a natural transformation of functors on $\gr^{op}$, it is sufficient to prove the naturality for these five homomorphisms. Using the explicit description of the functor $Cat\mathcal{A}ss^u(c,-):  \gr^{op}  \to \ev$, for $c \in \nat$, given in \cite[Lemma A.2]{P-21-2} and the definition of the composition in the category $\mathbf{A}$ given in \cite{HM}, we obtain that the maps $f_l$  define a natural transformation:
$${\underset{c \in \nat}{\bigoplus}} Cat\mathcal{A}ss^u  (c,-) \underset{\mathfrak{S}_c}{\otimes}{J}_d^{F_n}(c) \xrightarrow{f} \mathbf{A}_d(n,-)_{\mathbf{0} }.$$

For example, for $m_5: F_2 \to F_1$, the induced map $Cat\mathcal{A}ss^u(c,1) \to Cat\mathcal{A}ss^u(c,2)$ sends a set map $f:c\to 1$ with an order of $c$ to the sum of all the maps $c\to 2$ obtained by shuffles and the map $\mathbf{A}_d(n,-)_{\mathbf{0}}(1) \to \mathbf{A}_d(n,-)_{\mathbf{0}}(2)$ sends a Jacobi diagram on $X_1$ to the sum of the Jacobi diagrams on $X_1$ and $X_2$ obtained by a shuffle of the univalent vertices. This corresponds in $\mathbf{A}$ to the box notation used to define the composition.

Since $Cat\mathcal{L}ie $ is generated by the morphisms $\mu_i^{c+1} \in Cat\mathcal{L}ie (c+1,c)$, to prove that $f \circ L=f\circ R$ it is sufficient to prove this relation on these generators. Let $[\epsilon]$ be a generator in $Cat\mathcal{A}ss^u(c,l)$ represented by a set map $\epsilon:c\to l$ and a given order on each of its fiber. We denote by $E_i$ the ordered fiber of $\epsilon(i)$ by $\epsilon$: explicitly we have:
$$E_i=\{a_1<\ldots< a_u<i<b_1<\ldots <b_v\}.$$ We consider the following ordered sets $E_{i<m+1}=\{a_1<\ldots< a_u<i<m+1<b_1<\ldots <b_v\}$ and $E_{m+1<i}=\{a_1<\ldots< a_u<m+1<i<b_1<\ldots <b_v\}$. 

We have:
$$L_l([\epsilon] \otimes [\mu_i^{c+1}] \otimes [D])=[(\epsilon \circ s_i^{c+1}, E_{i<m+1})] \otimes [D]-[(\epsilon \circ s_i^{c+1}, E_{m+1<i})] \otimes [D] $$
where $[(\epsilon \circ s_i^{c+1}, E_{i<m+1})]$ is the generator in $Cat\mathcal{A}ss^u(c,l)$ represented by the set map $\epsilon \circ s_i^{c+1}:c+1\to l$ and the order on the fibers over $j$, is the same that for $\epsilon$  for $j \not= i$ and is $E_{i<m+1}$ for $j=i$. The generator $[(\epsilon \circ s_i^{c+1}, E_{m+1<i})]$ is defined similarly.

We have:
$$R_l([\epsilon] \otimes [\mu_i^{c+1}] \otimes [D])=[\epsilon] \otimes 		\vcenter{\hbox{\begin{tikzpicture}[baseline=1.8ex,scale=0.4]
	\draw[thick, dotted]   (2,1) to (2,0);
	\draw[thick, dotted]   (1,2) to (2,1);
	\draw[thick, dotted] (2,1) to (3,2);
	\draw  (0,2) to (4,2);
	\draw  (0,2) to (0,4);
	\draw  (0,4) to (4,4);
	\draw  (4,2) to (4,4);
	\draw (1,2) node[below] {$\scriptstyle{i}$};
	\draw (3.3,2) node[below] {$\scriptstyle{m+1}$};
		\draw (2,0) node[below] {$\scriptstyle{i}$};
		\draw (2,4) node[below] {$D$};
		\end{tikzpicture}}}
	$$
	By the STU relation we obtain that:
	$$f_l \circ L_l([\epsilon] \otimes [\mu_i^{c+1}] \otimes [D])=f_l \circ R_l([\epsilon] \otimes [\mu_i^{c+1}] \otimes [D])$$
	and we deduce that $f_l$ define a natural transformation
	$$Cat\mathcal{A}ss^u \underset{Cat\mathcal{L}ie}{\otimes } J_d^{F_n} \to \mathbf{A}_d(n,-)_{\mathbf{0} }.$$
	By the isomorphism given in (\ref{alpha2}), we obtain that this natural transformation is a natural equivalence.
\end{proof}

\begin{cor} \label{cor-thm}
For $n,d \in \mathbb{N}$, we have an equivalence of functors in $\mathcal{F}_{\omega}(\gr^{op}; \mathbb{K})$:
$$\alpha((J_d^{F_n})_{\leq l}) \simeq \mathbf{A}_d^{2d-l}(n,-)_{\mathbf{0} }.$$
\end{cor}
\begin{proof}
Since $((J_d^{F_n})_{\leq l})(i)=0$ for $i>l$, non-zero elements in $((J_d^{F_n})_{\leq l})(i)$ are open Jacobi diagrams, of degree $d$, having at most $l$ univalent vertices and so at least $2d-l$ trivalent vertices. So, by the equivalence of categories described in Theorem \ref{thm-J_d}, the subfunctor $(J_d^{F_n})_{\leq l}$ of $J_d^{F_n}$ corresponds to the subfunctor $\mathbf{A}_d^{2d-l}(n,-)_{\mathbf{0} }$ of $\mathbf{A}_d(n,-)_{\mathbf{0} }$.

\end{proof}

For $m=0$, since $\mathbf{A}_d(0,-)_{\mathbf{0} }=\mathbf{A}_d(0,-)$ it follows from Theorem \ref{thm-J_d} that  $\alpha(J_d^{\{1\}})=\mathbf{A}_d(0,-).$

In the rest of this section we will exploit the correspondance given in Theorem \ref{thm-J_d} in order to study the functors $\mathbf{A}_d(n,-)_{\mathbf{0} }$.

%%%%%%%%%%
\subsection{On the polynomial filtration of the functors $\mathbf{A}_d(n,-)_{\mathbf{0} }$}

\begin{thm} \label{poly_1}
For $n \in \mathbb{N}$ and $d\geq 1$, the functor $\mathbf{A}_d(n,-)_{\mathbf{0} }: \gr^{op} \to \ev$ is polynomial of degree $2d$ and the filtration of $\mathbf{A}_d(n,-)_{\mathbf{0} }$ given in Proposition \ref{filtration2} corresponds to the polynomial filtration. In other words
$$ \mathbf{p}_{2d-t}(\mathbf{A}_d(n,-)_{\mathbf{0} })=\mathbf{A}^t_{d}(n, -)_{\mathbf{0}}.$$
\end{thm}

\begin{proof}
Since $J^{F_n}_d(2d+1)=0$ and $J^{F_n}_d(2d)\not= 0$, the functor $\alpha(J^{F_n}_d) \in \mathcal{F}_{\omega}(\gr^{op};  \mathbb{K})$ is a polynomial functor of degree $2d$, by Section \ref{CatLie}.

By (\ref{filtration-general}), Theorem \ref{thm-J_d} and Corollary \ref{cor-thm} we have: 
$$\alpha^{-1}(\mathbf{p}_i(\mathbf{A}_d(n,-)_{\mathbf{0} }))\simeq (\alpha^{-1} (\mathbf{A}_d(n,-))_{\mathbf{0} })_{\leq i}\simeq (J_d^{F_n})_{\leq i}\simeq \alpha^{-1}(\mathbf{A}_d^{2t-i}(n,-)_{\mathbf{0} }).$$

\end{proof}

\begin{rem}
The polynomiality of the functors $\mathbf{A}_d(n,-)_{\mathbf{0} }$ can be proved without using the equivalence of categories $\f_\omega(\gr^{op};\mathbb{K}) \simeq \f_{\mathcal{L}ie}$, using similar arguments as in the proof of Lemma \ref{lm-d>0}. As this seems instructive to us, we give this alternative proof below.

We want to prove that $\widetilde{cr}_{2d+1}(\mathbf{A}_d(n,-)_{\mathbf{0} })=0$. By Section \ref{rappels-poly}, $\widetilde{cr}_{2d+1}(\mathbf{A}_d(n,-)_{\mathbf{0} })(1,\ldots, 1)$ is the cokernel of the following homomorphism:
$$\overset{2d+1}{\underset{k=1}{\bigoplus}}\mathbf{A}_d(n,-)_{\mathbf{0} }(F_{2d}) \xrightarrow{(\mathbf{A}_d(n,-)_{\mathbf{0} }(r^{2d+1}_{\hat{1}}), \ldots, \mathbf{A}_d(n,-)_{\mathbf{0} }(r^{2d+1}_{\hat{2d+1}}))} \mathbf{A}_d(n,-)_{\mathbf{0} }(F_{2d+1}) $$
A generator of $\mathbf{A}_d(n,-)_{\mathbf{0} }(F_{2d+1})$ is represented by a $F_n$-beaded Jacobi diagram $D$ on $X_{2d+1}$ having $2d$ vertices and without bead on $X_{2d+1}$. Since the Jacobi diagram has $2d$ univalent vertices, at least one of the $(2d+1)$-arc components of  $X_{2d+1}$ has no univalent vertex. Assume that this is the case for the $k$-th arc of $X_{2d+1}$. Denote by $D_{\hat{k}}$ the generator of   $\mathbf{A}_d(n,-)_{\mathbf{0} }(F_{2d})$ obtained from $D$ by forgetting the $k$-th arc of $X_{2d+1}$, then:
 $$\mathbf{A}_d(n,-)_{\mathbf{0} }(r^{2d+1}_{\hat{k}})(D_{\hat{k}}) =D.$$
 We deduce that the cokernel of the previous map is zero.
\end{rem}

Since $\mathbf{A}_d(0,-)_{\mathbf{0}}=\mathbf{A}_d(0,-)$, as a special case we obtain the following result, originally due to Katada.
\begin{cor} \label{poly-Katada} \cite[Proposition 8.1]{Katada}.
For $d\geq 1$, the functor $\mathbf{A}_d(0,-): \gr^{op} \to \ev$ is polynomial of degree $2d$.
\end{cor}

%%%%%%%%%%%%%%%%
\subsection{On the functors $\mathbf{A}_1(n,-)_{\mathbf{0} }$}
Recall that, in Section \ref{Action-S_2}, we define the functor $\mathcal{P}_2$ and study the action of $\mathfrak{S}_2$ on it.
\begin{prop} \label{A11}
For $n \in \nat$, we have a natural equivalence
$$\mathbf{A}_1(n,-)_{\mathbf{0} }\simeq \mathcal{P}_2^{\#} \underset{\mathfrak{S}_2}{\otimes} \mathbb{K}[F_n],$$
where the action of $\mathfrak{S}_2$ on $\mathbb{K}[F_n]$ is given by taking the inverse in $F_n$: $v\mapsto v^{-1}$ and the action of $\mathfrak{S}_2$ on $\mathcal{P}_2^{\#}$ is given in Section \ref{Action-S_2}.
In particular, we have $\mathbf{A}_1(0,-) \simeq S^2 \circ \A^{\#}$.
\end{prop}
The second part of the statement corresponds to a result of Katada given in \cite[Section 4]{Katada}. 

\begin{proof}[Proof of Proposition \ref{A11}]
By Theorem  \ref{thm-J_d}, the equivalence of categories $\alpha^{-1}: \f_\omega(\gr^{op};\mathbb{K}) \xrightarrow{\simeq} \f_{\mathcal{L}ie}$ and Corollary  \ref{cor-Passi}, the statement is equivalent to the existence of a natural equivalence:
$$J_1^{F_n}\simeq   Cat\mathcal{L}ie(2,-) \underset{\mathfrak{S}_2}{\otimes} \mathbb{K}[F_n].$$

The functor $J_1^{F_n}: Cat\mathcal{L}ie \to \ev$ is non-zero only on $1$ and $2$ and we have: 

$$J_1^{F_n}(2)=\mathbb{K}[ 
\vcenter{{\hbox{\begin{tikzpicture}[baseline=1.8ex,scale=0.4]
	\draw[thick,dotted] [->,>=latex]  (0,0) to (2,0);
	\draw (0,0) node[above] {$\scriptstyle{1}$};
	\draw (2,0) node[above] {$\scriptstyle{2}$};
	\draw (1,0) node[below] {$\scriptstyle{w}$};
		\draw[fill=white] (1,0) circle (4pt);
	\end{tikzpicture}}}}
] \simeq \mathbb{K}[F_n];$$
$$J_1^{F_n}(1)=\mathbb{K}[ 
\vcenter{{\hbox{\begin{tikzpicture}[baseline=1.8ex,scale=0.4]
	\draw[thick, dotted]  (0,0) to (0,1);
	\draw [thick, dotted]  (0,1) to [bend left=90]  (0,2);
	\draw[thick, dotted] [-To,>=latex]  (0,2) to   [bend left=25] (0.5,1.5);
	\draw [thick, dotted]  (0.5,1.5) to  [bend left=30]   (0,1);

	\draw (0,0) node[below] {$\scriptstyle{1}$};
		\draw (0,2) node[above] {$\scriptstyle{w}$};
		\draw[fill=white] (0,2) circle (4pt);
	\end{tikzpicture}}}}
\ ] /\langle 
\vcenter{{\hbox{\begin{tikzpicture}[baseline=1.8ex,scale=0.4]
	\draw[thick, dotted]  (0,0) to (0,1);
	\draw [thick, dotted]  (0,1) to [bend left=90]  (0,2);
	\draw[thick, dotted] [-To,>=latex]  (0,2) to   [bend left=25] (0.5,1.5);
	\draw [thick, dotted]  (0.5,1.5) to  [bend left=30]   (0,1);

	\draw (0,0) node[below] {$\scriptstyle{1}$};
		\draw (0,2) node[above] {$\scriptstyle{w}$};
		\draw[fill=white] (0,2) circle (4pt);
	\end{tikzpicture}}}}
	+
	\vcenter{{\hbox{\begin{tikzpicture}[baseline=1.8ex,scale=0.4]
	\draw[thick,  dotted]  (0,0) to (0,1);
\draw [thick, dotted]  (0,1) to [bend left=90]  (0,2);
	\draw[thick, dotted] [-To,>=latex]  (0,2) to   [bend left=25] (0.5,1.5);
	\draw [thick, dotted]  (0.5,1.5) to  [bend left=30]   (0,1);

	\draw (0,0) node[below] {$\scriptstyle{1}$};
		\draw (0,2) node[above] {$\scriptstyle{w}^{-1}$};
		\draw[fill=white] (0,2) circle (4pt);
	\end{tikzpicture}}}}
	\rangle
\simeq \mathbb{K}[F_n]/\langle [w]+[w^{-1}] \rangle;$$
 $\mathfrak{S}_2$ acts on $J_1^{F_n}(2)$ taking the inverse in $F_n$ and $J_1^{F_n}(\mu^2_1)([w])=\overline{[w]}$, for $w \in F_n$.

The functor $Cat\mathcal{L}ie(2,-): Cat\mathcal{L}ie \to \ev$ is non-zero only on $1$ and $2$ and we have:
 
$$Cat\mathcal{L}ie(2,2) \simeq  \mathbb{K}[\mathfrak{S}_2]\quad \text{and} \quad Cat\mathcal{L}ie(2,1) \simeq  \mathbb{K}[\mu^2_1];$$ 
and $Cat\mathcal{L}ie(2,-)(\mu^2_1)([\tau])=-[\mu^2_1]$ for $\tau$ the generator of $\mathfrak{S}_2$.

In order to define a natural transformation $\rho: Cat\mathcal{L}ie(2,-)\underset{\mathfrak{S}_2}{\otimes} \mathbb{K}[F_n] \to J_1^{F_n}$, we define the $\mathbb{K}$-linear maps:
\begin{itemize}
\item $\rho_2: \mathbb{K}[\mathfrak{S}_2]\underset{\mathfrak{S}_2}{\otimes} \mathbb{K}[F_n] \to \mathbb{K}[F_n]$ given by $\rho_2([\sigma] \otimes [w])=\sigma.[w]$;
\item $\rho_1: \mathbb{K}[\mu^2_1] \underset{\mathfrak{S}_2}{\otimes} \mathbb{K}[F_n] \to \mathbb{K}[F_n]/\langle [w]+[w^{-1}] \rangle$ given by $\rho_1([\mu^2_1] \otimes [w])=\overline{[w]}$.
\end{itemize}
Denoting by $C$ the functor $Cat\mathcal{L}ie(2,-)\underset{\mathfrak{S}_2}{\otimes} \mathbb{K}[F_n] $, we have:
$$\rho_2 \circ C(\tau)([\tau] \otimes [w])=\rho_2([\tau \circ \tau] \otimes [w])=\rho_2([Id]\otimes [w])=[w]$$
$$ \quad \text{and} \quad J_1^{F_n}(\tau) \circ \rho_2 ([\tau] \otimes [w])=J_1^{F_n}(\tau) (\tau.[w])=J_1^{F_n}(\tau) ([w^{-1}])=[w]$$
and
$$\rho_1 \circ C(\mu_1^2)([\tau]\otimes [w])=\rho_1(-[\mu_1^2] \otimes [w])=-\overline{[w]}$$
$$\quad \text{and} \quad J_1^{F_n}(\mu_1^2) \circ \rho_2([\tau]\otimes [w])=J_1^{F_n}(\mu_1^2) (\tau.[w])=J_1^{F_n}(\mu_1^2)([w^{-1}])=\overline{[w^{-1}]}=-\overline{[w]}.$$
By similar computations on the generators $[Id] \otimes [w]$, we obtain that $\rho_1$ and $\rho_2$ satisfy the two relations
$$\rho_2 \circ C(\tau)= J_1^{F_n}(\tau) \circ \rho_2 \qquad \text{and} \qquad \rho_1 \circ C(\mu_1^2)=J_1^{F_n}(\mu_1^2) \circ \rho_2$$
and so define a natural transformation.

Since $\rho_1$ and $\rho_2$ are isomorphisms, $\rho$ is a natural equivalence.

For $n=0$, $ \mathbb{K}[F_0]/\langle [w]+[w^{-1}] \rangle =0$, so the functor $Cat\mathcal{L}ie(2,-)\underset{\mathfrak{S}_2}{\otimes} \mathbb{K}[F_n] $ is the atomic functor $\mathbb{K}[2]$ and by (\ref{atomic}) we have:
$$ \alpha(J_1^{F_0}) = \alpha(\mathbb{K}[2]) \simeq (\A^{\#})^{\otimes 2} \underset{\mathfrak{S}_2}{\otimes} \mathbb{K}\simeq ( (\A^{\#})^{\otimes 2} )^{\mathfrak{S}_2} \simeq S^2 \circ \A^{\#}.$$

\end{proof}

\subsection{Outer property of the functors $\mathbf{A}_d(n,-)_{\mathbf{0} }$}

For $d \in \nat$, in \cite[Theorem 5.1]{Katada}, Katada proves that $\mathbf{A}_d(0,-)$ is an outer functor, namely inner automorphisms of $F_m$ act trivially on $\mathbf{A}_d(0,m)$. Her proof is based on properties of the composition in the category $\mathbf{A}$, especially properties of the box operator. In Theorem \ref{Thm-outre}, we study the outer property of the functors $\mathbf{A}_d(n,-)_{\mathbf{0} }$ using the equivalence of categories $ \f_\omega(\gr^{op};\mathbb{K}) \simeq  \f_{\mathcal{L}ie}$. For $n=0$, this gives another proof of Katada's result.

\begin{thm} \label{Thm-outre}
For $d,n \in \nat$, the functor $\mathbf{A}_d(n,-)_{\mathbf{0} }$ is an outer functor iff $n=0$ or $d=0$.
\end{thm}
By Theorem \ref{thm-J_d} and Section \ref{Rappels-G}, $\mathbf{A}_d(n,-)_{\mathbf{0} }$ is an outer functor iff $J_d^{F_n}$ belongs to $\mathcal{F}_{Lie}^{\mu}$. 

The proof of this theorem relies on the following result.
\begin{prop}\label{Prop-outre}  
The functor $J_d^{F_0}$ belongs to $\mathcal{F}_{Lie}^{\mu}$. 
\end{prop}

\begin{proof}
For simplicity, the functor $J_d^{F_0}$ is denoted by $J_d$ below.\\
For $k \in \mathbb{N}$, the natural transformation $(\mu_{J_d}): \tau J_d \to J_d$ gives maps
$$(\mu_{J_d})_{k}: \tau J_d(k)=J_d(k+1) \to J_d(k).$$

\begin{itemize}
\item For $k\geq 2d$, $J_d(k+1)=0$ so $(\mu_{J_d})_{k}=0$.
\item For $k=2d-1$, the generators of $J_d(2d)$ are Jacobi diagrams of the form 
\[
D:=	\vcenter{\hbox{\begin{tikzpicture}[baseline=1.8ex,scale=0.4]
	\draw[thick, dotted]   (0,3) to (0,0);
	\draw[thick, dotted]  (2,3) to (2,0);
	\draw[thick, dotted]  (6,3) to (6,0);
	\draw[thick, dotted]  (8,3) to (8,0);
		\draw (4,2) node[below] {$\scriptstyle{\ldots}$};
		\draw (0,0) node[below] {$\scriptstyle{j}_1$};
	\draw (6,0) node[below] {$\scriptstyle{j_{d-1}}$};
		\draw (6,3) node[above] {$\scriptstyle{i_{d-1}}$};
		\draw (0,3) node[above] {$\scriptstyle{i}_1$};
		\draw (2,0) node[below] {$\scriptstyle{j_2}$};
		\draw (2,3) node[above] {$\scriptstyle{i_2}$};
		\draw (8,3) node[above] {$\scriptstyle{\beta}$};
		\draw (8,0) node[below] {$\scriptstyle{2d}$};
		\end{tikzpicture}}}
		\]
		For $d=1$, we have
 $$(\mu_{J_1})_{1}(D)=(\mu_{J_1})_{1}(
 \vcenter{{\hbox{\begin{tikzpicture}[baseline=1.8ex,scale=0.4]
	\draw[thick, dotted]  (0,0) to (2,0);
	\draw (0,0) node[above] {$\scriptstyle{1}$};
	\draw (2,0) node[above] {$\scriptstyle{2}$};
	
	\end{tikzpicture}}}})
	=
	\vcenter{{\hbox{\begin{tikzpicture}[baseline=1.8ex,scale=0.4]
	\draw[thick, dotted]  (0,0) to (0,1);
	\draw[thick, dotted]   (0,1) to [bend left=90]  (0,2);
	\draw[thick, dotted]  (0,2) to [bend left=90]   (0,1);
	\draw (0,0) node[below] {$\scriptstyle{1}$};
	\end{tikzpicture}}}} 
		=0 \text{\ by\ the\ AS\ relation}.$$

For $d>1$, and $\alpha \in \{1, \ldots, d-1\}$
\[
\mu_{i_\alpha}^{2d}(D)=	\vcenter{\hbox{\begin{tikzpicture}[baseline=1.8ex,scale=0.4]
	\draw[thick, dotted]   (2,3) to (2,0);
	\draw[thick, dotted]   (6,3) to (6,0);
	\draw[thick, dotted] (-1,1) to (-2,3);
	\draw[thick, dotted]   (-1,1) to (0,3);
	\draw[thick, dotted]   (-1,1) to (-1,0);
	\draw (-1,0) node[below] {$\scriptstyle{i_\alpha}$};
	\draw (4,2) node[below] {$\scriptstyle{\ldots}$};
	\draw (-2,3) node[above] {$\scriptstyle{j_\alpha}$};
	\draw (0,3) node[above] {$\scriptstyle{\beta}$};
		\end{tikzpicture}}}
	\qquad 	\text{\ and\ } \qquad 
\mu_{j_\alpha}^{2d}(D)=	\vcenter{\hbox{\begin{tikzpicture}[baseline=1.8ex,scale=0.4]
	\draw[thick, dotted]   (2,3) to (2,0);
	\draw[thick, dotted]   (6,3) to (6,0);
	\draw[thick, dotted] (-1,1) to (-2,3);
	\draw[thick, dotted]   (-1,1) to (0,3);
	\draw [thick, dotted]  (-1,1) to (-1,0);
	\draw (-1,0) node[below] {$\scriptstyle{j_\alpha}$};
	\draw (4,2) node[below] {$\scriptstyle{\ldots}$};
	\draw (-2,3) node[above] {$\scriptstyle{i_\alpha}$};
	\draw (0,3) node[above] {$\scriptstyle{\beta}$};
		\end{tikzpicture}}}
	\]

Using the AS relation we have $\mu_{i_\alpha}^{2d}(D)+\mu_{j_\alpha}^{2d}(D)=0$ and
$\mu_{\beta}^{2d}(D)=0$. We deduce that $(\mu_{J_d})_{2d-1}=0$.
\item For $k<2d-1$. By functoriality of $J_d$ on $Cat\mathcal{L}ie$, we have morphisms
\begin{equation} \label{surj-J}
Cat\mathcal{L}ie(k+2,k+1) \otimes J_d(k+2) \to J_d(k+1)
\end{equation}
We will prove that these maps are surjective.

Let $D$ be a generator of $J_d(k+1)$; since $k+1<2d$, $D$ has at least one connected component which is not of the form $  \vcenter{{\hbox{\begin{tikzpicture}[baseline=1.8ex,scale=0.4]
	\draw[thick, dotted]  (0,1) to (2,1);
		\end{tikzpicture}}}}$. We can chose one of these connected components having the form 
	\[	
		\vcenter
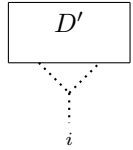
{\hbox{\begin{tikzpicture}[baseline=1.8ex,scale=0.4]
	\draw[thick, dotted]   (2,1) to (2,0);
	\draw[thick, dotted]   (1,2) to (2,1);
	\draw[thick, dotted] (2,1) to (3,2);
	\draw  (0,2) to (4,2);
	\draw  (0,2) to (0,4);
	\draw  (0,4) to (4,4);
	\draw  (4,2) to (4,4);
		\draw (2,0) node[below] {$\scriptstyle{i}$};
		\draw (2,4) node[below] {$D'$};
		\end{tikzpicture}}}
	\]
	where $i \in \{1, \ldots, k+1\}$ and $D'$ is a Jacobi diagram.
The generator $D$ is obtained by applying $\mu_i^{k+2}$ to the generator of $J_d(k+2)$ obtained from $D$ replacing the previous connected component by \[	
		\vcenter
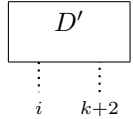
{\hbox{\begin{tikzpicture}[baseline=1.8ex,scale=0.4]
	\draw[thick, dotted]   (1,2) to (1,1);
	\draw[thick, dotted] (3,1) to (3,2);
	\draw  (0,2) to (4,2);
	\draw  (0,2) to (0,4);
	\draw  (0,4) to (4,4);
	\draw  (4,2) to (4,4);
		\draw (1,1) node[below] {$\scriptstyle{i}$};
		\draw (3,1) node[below] {$\scriptstyle{k+2}$};
		\draw (2,4) node[below] {$D'$};
		\end{tikzpicture}}}
	\]
	(which could be non-connected).
	By iteration we obtain that the morphism
	$$Cat\mathcal{L}ie(2d,k+1) \otimes J_d(2d) \to J_d(k+1)$$
	is surjective. 
	By the naturality of $\mu_{J_d}$ and the fact that $(\mu_{J_d})_{2d-1}=0$, we deduce that $(\mu_{J_d})_{k}=0$.	

\end{itemize}

\end{proof}

\begin{proof}[Proof of Theorem  \ref{Thm-outre}]
For $d=0$, by (\ref{A_0}) $\mathbf{A}_0(n,-)_{\mathbf{0} } \simeq \mathbb{K}$ which  is obviously an outer functor. 

For $n=0$, $\mathbf{A}_d(0,-)_{\mathbf{0} }$ is an outer functor by Proposition \ref{Prop-outre}.

If $n\not=0$ and $d \geq 1$,  we prove that $J_d^{F_n}$ does not belong to $\mathcal{F}_{Lie}^{\mu}$. 
The natural transformation $(\mu_{J_d^{F_n}}): \tau J_d^{F_n} \to J_d^{F_n}$ gives map
$$(\mu_{J_d^{F_n}})_{2d-1}: \tau J_d^{F_n}(2d-1)=J_d^{F_n}(2d) \to J_d^{F_n}(2d-1).$$
Consider the following generator of $J_d^{F_n}(2d)$
\[
D:=	\vcenter{\hbox{\begin{tikzpicture}[baseline=1.8ex,scale=0.4]
	\draw[thick, dotted] [->,>=latex]  (0,3) to (0,0);
	\draw[thick, dotted] [->,>=latex]  (2,3) to (2,0);
	\draw[thick, dotted] [->,>=latex]  (6,3) to (6,0);
	\draw (0,0) node[below] {$\scriptstyle{d+1}$};
	\draw (2,0) node[below] {$\scriptstyle{d+2}$};
	\draw (6,0) node[below] {$\scriptstyle{2d}$};
	\draw (4,2) node[below] {$\scriptstyle{\ldots}$};
	\draw (0,3) node[above] {$\scriptstyle{1}$};
	\draw (2,3) node[above] {$\scriptstyle{2}$};
	\draw (6,3) node[above] {$\scriptstyle{d}$};
	\draw (6,1) node[right] {$\scriptstyle{w_d}$};
	\draw (2,1) node[right] {$\scriptstyle{w_2}$};
	\draw (0,1) node[right] {$\scriptstyle{w_1}$};
		\draw[fill=white] (0,1) circle (4pt);
		\draw[fill=white] (2,1) circle (4pt);
		\draw[fill=white] (6,1) circle (4pt);
	\end{tikzpicture}}}
		\]
where $w_1, \ldots, w_d \in F_n$.

For $d=1$, we have
 $$(\mu_{J_1^{F_n}})_{1}(D)=(\mu_{J_1^{F_n}})_{1}(
 \vcenter{{\hbox{\begin{tikzpicture}[baseline=1.8ex,scale=0.4]
	\draw[thick, dotted] [-To,>=latex]  (0,0) to (2,0);
	\draw (0,0) node[above] {$\scriptstyle{1}$};
	\draw (2,0) node[above] {$\scriptstyle{2}$};
	\draw (1,0) node[below] {$\scriptstyle{w_1}$};
		\draw[fill=white] (1,0) circle (4pt);
	\end{tikzpicture}}}})
	=
	\vcenter{{\hbox{\begin{tikzpicture}[baseline=1.8ex,scale=0.4]
	\draw[thick, dotted]  (0,0) to (0,1);
\draw [thick, dotted]  (0,1) to [bend left=90]  (0,2);
	\draw[thick, dotted] [-To,>=latex]  (0,2) to   [bend left=25] (0.5,1.5);
	\draw [thick, dotted]  (0.5,1.5) to  [bend left=30]   (0,1);
	\draw (0,0) node[below] {$\scriptstyle{1}$};
		\draw (0,2) node[above] {$\scriptstyle{w}_1$};
		\draw[fill=white] (0,2) circle (4pt);
	\end{tikzpicture}}}} 
		.$$
Since $n\geq 1$, for $w_1\not=1 \in F_n$, $(\mu_{J_1^{F_n}})_{1}(D)\not=0$, so $J_1^{F_n}$ is not an outer $Cat\mathcal{L}ie$-module.
		
For $d>1$,
 $(\mu_{J_d^{F_n}})_{2d-1}(D)$ is a sum of Jacobi diagrams of the form 
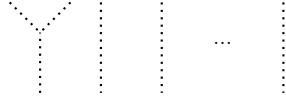
\[
\vcenter{\hbox{\begin{tikzpicture}[baseline=1.8ex,scale=0.4]
	\draw[thick, dotted]  (0,3) to (0,0);
	\draw [thick, dotted] (2,3) to (2,0);
	\draw [thick, dotted]  (6,3) to (6,0);
	\draw [thick, dotted]  (-3,3) to (-2,2);
	\draw [thick, dotted]  (-1,3) to (-2,2);
	\draw[thick, dotted]   (-2,0) to (-2,2);
	\draw (4,2) node[below] {$\scriptstyle{\ldots}$};
		\end{tikzpicture}}}
		\]
		with a $2d$-labelling and beads in $F_n$ on edges. In this sum there are exactly two summands where the tree 
		${\hbox{\begin{tikzpicture}[baseline=1.8ex,scale=0.4]
	\draw[thick, dotted]   (-3,3) to (-2,2);
	\draw[thick, dotted]   (-1,3) to (-2,2);
	\draw[thick, dotted]   (-2,0) to (-2,2);
		\end{tikzpicture}}}$ is labelled by the set $\{1,d,d+1\}$:
	\[
\mu_1^{2d}(D)=	\vcenter{\hbox{\begin{tikzpicture}[baseline=1.8ex,scale=0.4]
	\draw[thick, dotted] [->,>=latex]  (2,3) to (2,0);
	\draw[thick, dotted] [->,>=latex]  (6,3) to (6,0);
	\draw[thick, dotted] [->,>=latex]  (-1,1) to (-2,3);
	\draw[thick, dotted] [->,>=latex]  (-1,1) to (0,3);
	\draw [thick, dotted]  (-1,1) to (-1,0);
	\draw (-1,0) node[below] {$\scriptstyle{1}$};
	\draw (2,0) node[below] {$\scriptstyle{d+2}$};
	\draw (6,0) node[below] {$\scriptstyle{2d-1}$};
	\draw (4,2) node[below] {$\scriptstyle{\ldots}$};
	\draw (2,3) node[above] {$\scriptstyle{2}$};
	\draw (-2,3) node[above] {$\scriptstyle{d+1}$};
	\draw (0,3) node[above] {$\scriptstyle{d}$};
	\draw (6,3) node[above] {$\scriptstyle{d-1}$};
	\draw (6,1) node[right] {$\scriptstyle{w_{d-1}}$};
	\draw (2,1) node[right] {$\scriptstyle{w_2}$};
	\draw (-1.5,2) node[left] {$\scriptstyle{w_1}$};
	\draw (-0.5,2) node[right] {$\scriptstyle{w_d^{-1}}$};
		\draw[fill=white] (2,1) circle (4pt);
		\draw[fill=white] (6,1) circle (4pt);
		\draw[fill=white] (-0.5,2) circle (4pt);
		\draw[fill=white] (-1.5,2) circle (4pt);
	\end{tikzpicture}
	=
	- \begin{tikzpicture}[baseline=1.8ex,scale=0.4]
	\draw[thick, dotted] [->,>=latex]  (2,3) to (2,0);
	\draw[thick, dotted] [->,>=latex]  (6,3) to (6,0);
	\draw[thick, dotted] [->,>=latex]  (-1,1) to (-2,3);
	\draw[thick, dotted] [->,>=latex]  (-1,1) to (0,3);
	\draw [thick, dotted]  (-1,1) to (-1,0);
	\draw (-1,0) node[below] {$\scriptstyle{1}$};
	\draw (2,0) node[below] {$\scriptstyle{d+2}$};
	\draw (6,0) node[below] {$\scriptstyle{2d-1}$};
	\draw (4,2) node[below] {$\scriptstyle{\ldots}$};
	\draw (2,3) node[above] {$\scriptstyle{2}$};
	\draw (-2,3) node[above] {$\scriptstyle{d}$};
	\draw (0,3) node[above] {$\scriptstyle{d+1}$};
	\draw (6,3) node[above] {$\scriptstyle{d-1}$};
	\draw (6,1) node[right] {$\scriptstyle{w_{d-1}}$};
	\draw (2,1) node[right] {$\scriptstyle{w_2}$};
	\draw (-1.5,2) node[left] {$\scriptstyle{w_d^{-1}}$};
	\draw (-0.5,2) node[right] {$\scriptstyle{w_1}$};
		\draw[fill=white] (2,1) circle (4pt);
		\draw[fill=white] (6,1) circle (4pt);
		\draw[fill=white] (-0.5,2) circle (4pt);
		\draw[fill=white] (-1.5,2) circle (4pt);
	\end{tikzpicture}}}	
		\]
and		
		
	\[
\mu_{d+1}^{2d}(D)=	\vcenter{\hbox{\begin{tikzpicture}[baseline=1.8ex,scale=0.4]
	\draw[thick, dotted] [->,>=latex]  (2,3) to (2,0);
	\draw[thick, dotted] [->,>=latex]  (6,3) to (6,0);
	\draw[thick, dotted] [->,>=latex]  (-1,1) to (-2,3);
	\draw[thick, dotted] [->,>=latex]  (-1,1) to (0,3);
	\draw[thick, dotted]   (-1,1) to (-1,0);
	\draw (-1,0) node[below] {$\scriptstyle{d+1}$};
	\draw (2,0) node[below] {$\scriptstyle{d+2}$};
	\draw (6,0) node[below] {$\scriptstyle{2d-1}$};
	\draw (4,2) node[below] {$\scriptstyle{\ldots}$};
	\draw (2,3) node[above] {$\scriptstyle{2}$};
	\draw (-2,3) node[above] {$\scriptstyle{1}$};
	\draw (0,3) node[above] {$\scriptstyle{d}$};
	\draw (6,3) node[above] {$\scriptstyle{d-1}$};
	\draw (6,1) node[right] {$\scriptstyle{w_{d-1}}$};
	\draw (2,1) node[right] {$\scriptstyle{w_2}$};
	\draw (-1.5,2) node[left] {$\scriptstyle{w_1^{-1}}$};
	\draw (-0.5,2) node[right] {$\scriptstyle{w_d^{-1}}$};
		\draw[fill=white] (2,1) circle (4pt);
		\draw[fill=white] (6,1) circle (4pt);
		\draw[fill=white] (-0.5,2) circle (4pt);
		\draw[fill=white] (-1.5,2) circle (4pt);
	\end{tikzpicture}
	=
	\begin{tikzpicture}[baseline=1.8ex,scale=0.4]
	\draw[thick, dotted] [->,>=latex]  (2,3) to (2,0);
	\draw[thick, dotted] [->,>=latex]  (6,3) to (6,0);
	\draw[thick, dotted] [->,>=latex]  (-1,1) to (-2,3);
	\draw[thick, dotted] [->,>=latex]  (-1,1) to (0,3);
	\draw[thick, dotted]   (-1,1) to (-1,0);
	\draw (-1,0) node[below] {$\scriptstyle{1}$};
	\draw (2,0) node[below] {$\scriptstyle{d+2}$};
	\draw (6,0) node[below] {$\scriptstyle{2d-1}$};
	\draw (4,2) node[below] {$\scriptstyle{\ldots}$};
	\draw (2,3) node[above] {$\scriptstyle{2}$};
	\draw (-2,3) node[above] {$\scriptstyle{d}$};
	\draw (0,3) node[above] {$\scriptstyle{d+1}$};
	\draw (6,3) node[above] {$\scriptstyle{d-1}$};
	\draw (6,1) node[right] {$\scriptstyle{w_{d-1}}$};
	\draw (2,1) node[right] {$\scriptstyle{w_2}$};
	\draw (-1.5,2) node[left] {$\scriptstyle{w_1w_d^{-1}}$};
	\draw (-0.5,2) node[right] {$\scriptstyle{w_1}$};
		\draw[fill=white] (2,1) circle (4pt);
		\draw[fill=white] (6,1) circle (4pt);
		\draw[fill=white] (-0.5,2) circle (4pt);
		\draw[fill=white] (-1.5,2) circle (4pt);
	\end{tikzpicture}}	
	}
		\]
		where for $\mu_1^{2d}(D)$ we use the AS relation and for $\mu_{d+1}^{2d}(D)$ we use the relation in $F_n$-beaded Jacobi diagrams.

		For $w_1 \not= 1$, $\mu_1^{2d}(D)+ \mu_{d+1}^{2d}(D) \neq 0$. We deduce that, for $n\geq 1$, $(\mu_{J_d^{F_n}})_{2d-1}\neq 0$ and $J_d^{F_n}$ is not an outer $Cat\mathcal{L}ie$-module.

\end{proof}

\begin{rem}
We can also show that $\mathbf{A}_1(n,-)_{\mathbf{0} }$ is not an outer functor by using Proposition \ref{A11} and \cite[Example 11.13]{PV} where it is proved that the functor $\mathcal{P}_2$ is not an outer functor.
\end{rem}

%%%%%%%%%%%%%%%%%%%%%%
\subsection{On the functors $\mathbf{A}_d(0,-)$} \label{Ad0}
By Section \ref{rappels-poly}, we can consider the polynomial filtration of $\mathbf{A}_d(0,-)$ and by Proposition \ref{poly_1}, $ \mathbf{p}_{2d-i}(\mathbf{A}_d(0,-))=\mathbf{A}_d^{i}(0,-)$. So, the quotient 
$$ \mathbf{p}_{2d-i}(\mathbf{A}_d(0,-))/ \mathbf{p}_{2d-i-1}(\mathbf{A}_d(0,-))= \mathbf{A}_d^{i}(0,-)/\mathbf{A}_d^{i+1}(0,-)$$ corresponds to the functor denoted by $B_{d,i}$ in \cite{Katada}. 

By Corollary \ref{cor-thm}, $\alpha^{-1}( \mathbf{A}_d^{i}(0,-))\simeq (J_d^{F_0})_{\leq 2d-i}$, so 
$\alpha^{-1}\left( \mathbf{A}_d^{i}(0,-)/\mathbf{A}_d^{i+1}(0,-)\right)$
 is the atomic functor concentrated  in $2d-i$, where it is equal to the vector space $D_{2d-i}$ which is the quotient by AS and IHX relations of the $\mathbb{K}$-vector space generated by the $(2d-i)$-labelled Jacobi diagrams of degree $d$. The symmetric group $\mathfrak{S}_{2d-i}$ acts on $D_{2d-i}$ by the permutation of the labels of univalent vertices.
 
 By (\ref{atomic}) we obtain:
 $$B_{d,i}\simeq (\A^{\#})^{\otimes 2d-i} \underset{\mathfrak{S}_{2d-i}}{\otimes} D_{2d-i}$$
 corresponding to the description of the functor $B_{d,i}$ given by Katada in \cite[(3)]{Katada}.

The decomposition of $B_{d,0}$ given by Katada in \cite[Proposition 7.7]{KatadaII} is functorial. In other words, denoting by  $\mathbb{S}_{\lambda}$ the Schur functor associated with the partition $\lambda\vdash d$, we have, for $d\geq 0$:
\begin{equation}\label{B_d0}
B_{d,0}\simeq \underset{\lambda \vdash d}{\bigoplus} \mathbb{S}_{2 \lambda} \circ \A^{\#}
\end{equation}
where, for $\lambda=(\lambda_1, \ldots, \lambda_l)\vdash d$, $2\lambda$ is the partition $(2\lambda_1, \ldots,2\lambda_l)\vdash 2d$.

By \cite[Theorem 4.2]{V_ext}, $Ext^1_{\f(\gr;\mathbb{K})}(F,\mathbb{S}_{2d} \circ \A)=0$, for $F$ a polynomial functor so $\mathbb{S}_{2d} \circ \A$ is an injective object in the category of polynomial functors on $\gr$, so $\mathbb{S}_{2d} \circ \A^{\#}$ is a projective object in the category of polynomial functors on $\gr^{op}$. 
This allows us to give another proof of \cite[Theorem 10.1]{KatadaII}.
\begin{prop}\cite[Theorem 10.1]{KatadaII} \label{direct-decomposition}
For $d \in \nat$, we have a direct decomposition in $\f(\gr^{op}; \mathbb{K})$:
$$\mathbf{A}_d(0,-)= \mathbb{S}_{2d} \circ \A^{\#} \oplus \mathbf{A}_d(0,-)/\mathbb{S}_{2d} \circ \A^{\#}.$$
\end{prop}
\begin{proof}
By the polynomial filtration and (\ref{B_d0}), we have an epimorphism in $\f_{2d}(\gr^{op}; \mathbb{K})$:
\begin{equation} \label{Epi}
p: \mathbf{A}_d(0,-) \twoheadrightarrow  \mathbb{S}_{2d} \circ \A^{\#}.
\end{equation}
Since $\mathbb{S}_{2d} \circ \A^{\#}$ is a projective object in $\f_{2d}(\gr^{op}; \mathbb{K})$, the functor $\text{Hom}_{\f_{2d}(\gr^{op}; \mathbb{K})}(\mathbb{S}_{2d} \circ \A^{\#}, -): \f_{2d}(\gr^{op}; \mathbb{K}) \to \mathbf{Ab}$ is exact. Hence it sends the epimorphism (\ref{Epi}) to an epimorphism:
$$\text{Hom}_{\f_{2d}(\gr^{op}; \mathbb{K})}(\mathbb{S}_{2d} \circ \A^{\#}, \mathbf{A}_d(0,-))\twoheadrightarrow  \text{Hom}_{\f_{2d}(\gr^{op}; \mathbb{K})}(\mathbb{S}_{2d} \circ \A^{\#}, \mathbb{S}_{2d} \circ \A^{\#})$$
We deduce that $p$ has a section $s$, i.e. a natural transformation $s: \mathbb{S}_{2d} \circ \A^{\#} \to  \mathbf{A}_d(0,-)$ in $\f_{2d}(\gr^{op}; \mathbb{K})$  such that $p \circ s=Id_{\mathbb{S}_{2d} \circ \A^{\#}}$.
\end{proof}

Note that \cite[Proposition 10.2]{KatadaII} proves a stronger result, namely that $\mathbf{A}_d(0,-)/\mathbb{S}_{2d} \circ \A^{\#}$ is indecomposable.

 %%%%%%%%%%%%%%%%%%%%%%%%%%%%%%%%%%%%%%%%%%%%%%%%%%%%%%%%%%%%%%%%%%%%%%%%%%%%%%%%
%\nocite{*}
\bibliographystyle{amsalpha}
\bibliography{Jacobi.bib}
\end{document}